\documentclass[final,siam10,10pt]{siamltex1213}
\usepackage{latexsym,enumerate}
 \usepackage{mathrsfs}
 \usepackage{amsfonts}
 \usepackage{amssymb,amsmath}
 \usepackage{epsfig}       
 \usepackage{subfigure}    
 \usepackage{graphicx}
 \usepackage{float}
 \usepackage{multirow}
 \usepackage{array}
 \usepackage{color}
 \DeclareMathOperator*{\argmin}{argmin}
 \def\reff#1{{\rm(\ref{#1})}}
 
 \def \R{{\mathbb{R}}}
 \def \dom{{{\rm dom}~}}
 
\def\argmin{\mathop{\rm arg\,min}}
 
 \newtheorem{assumption}{Assumption}[section]

\title{Linear Convergence of proximal gradient algorithm with extrapolation for a class of nonconvex nonsmooth minimization problems}
\author{Bo Wen\footnotemark[2] \and Xiaojun Chen\footnotemark[3] \and Ting Kei Pong\footnotemark[3]}

\begin{document}
\maketitle

\renewcommand{\thefootnote}{\fnsymbol{footnote}}
\footnotetext[2]{Department of Mathematics, Harbin Institute of Technology, Harbin, P.R. China. Current address: Department of Applied Mathematics, The Hong Kong Polytechnic University, Hong Kong, P.R. China. (bo.wen@ connect.polyu.hk). This author's work was supported in part by the NSFC 11471088 and Hong Kong Research Grants Council PolyU5002/13p.}
\footnotetext[3]{Department of Applied Mathematics, The Hong Kong Polytechnic University, Hong Kong, P.R. China.(\email{xiaojun.chen@polyu.edu.hk}, \email{tk.pong@polyu.edu.hk}). The second author's work was supported in part by Hong Kong Research Grants Council PolyU153001/14p. The third author's work was supported in part by Hong Kong Research Grants Council PolyU253008/15p.}
\renewcommand{\thefootnote}{\arabic{footnote}}

\begin{abstract}
In this paper, we study the proximal gradient algorithm with extrapolation for minimizing the sum of a Lipschitz differentiable function and a proper closed convex function. Under the error bound condition used in \cite{Luo1993} for analyzing the convergence of the proximal gradient algorithm, we show that there exists a threshold such that if the extrapolation coefficients are chosen below this threshold, then the sequence generated converges $R$-linearly to a stationary point of the problem. Moreover, the corresponding sequence of objective values is also $R$-linearly convergent. In addition, the threshold reduces to $1$ for convex problems and, as a consequence, we obtain the $R$-linear convergence of the sequence generated by FISTA with fixed restart. Finally, we present some numerical experiments to illustrate our results.

\end{abstract}



\pagestyle{myheadings}
\thispagestyle{plain}
\markboth{PROXIMAL GRADIENT ALGORITHM WITH EXTRAPOLATION}{BO WEN, XIAOJUN CHEN, AND TING KEI PONG}

\section{Introduction}

In this paper, we consider the following optimization problem:
\begin{equation}\label{compro}
  \min_{x\in \R^n} F(x) := f(x) + g(x),
\end{equation}
where $g$ is a proper closed convex function and $f$ is a possibly nonconvex function that has a Lipschitz continuous gradient. We also assume that the proximal operator of $\mu g$, i.e.,
\begin{equation*}
  u \mapsto \argmin_{x\in \R^n} \left\{g(x) + \frac1{2\mu}\|x - u\|^2\right\}
\end{equation*}
is easy to compute for all $\mu > 0$ and any $u\in \R^n$, where $\argmin$ denotes the {\em unique} minimizer. We also assume that the optimal value of \eqref{compro} is finite and is attained. Problem~\eqref{compro} arises in many important contemporary applications including compressed sensing \cite{Tao2005,Donoho2006}, matrix completion \cite{CaRe09} and image processing \cite{Chambolle2004}. Since the problem instances are typically of large scale, first-order methods such as the proximal gradient algorithm \cite{Lions1979} are used for solving them, whose main computational efforts per iteration are the evaluations of the gradient of $f$ and the proximal mapping of $\mu g$. For the proximal gradient algorithm, when $f$ is in addition convex, it is known that
\[
F(x^k) - \inf_{x\in \R^n} F(x) = O\left(\frac1k\right),
\]
where $\{x^k\}$ is generated by the proximal gradient algorithm; see, for example, \cite[Theorem~1(a)]{Tseng2010}. However, the proximal gradient algorithm, in its original form, can be slow in practice; see, for example, \cite[Section 5]{Donoghue2013}.

Various attempts have thus been made to accelerate the proximal gradient algorithm. One simple and often efficient strategy is to perform extrapolation, where {\em momentum} terms involving the previous iterations are added to the current iteration. A prototypical algorithm takes the following form
\begin{eqnarray}\label{nest1}
\begin{cases}
y^{k}=x^{k}+\beta_{k}(x^{k}-x^{k-1}),\\
x^{k+1}=\argmin\limits_{x\in \R^n} \left\{\langle\nabla f(y^k),x\rangle + \frac1{2\mu}\|x - y^k\|^2 + g(x)\right\},
\end{cases}
\end{eqnarray}
where $\mu > 0$ is a constant that depends on the Lipschitz continuity modulus of $\nabla f$, and the extrapolation coefficients $\beta_k$ satisfy $0\le \beta_k \le 1$ for all $k$.
A recent example is the fast iterative shrinkage-thresholding algorithm (FISTA) proposed by Beck and Teboulle \cite{Beck2009}, which is based on Nesterov's extrapolation techniques \cite{Nesterov1983,Nesterov2004,Nesterov2005,Nesterov20071} and is designed for solving \eqref{compro} with $f$ being convex and $g$ being continuous. Their analysis can be directly extended to the case when $g$ is a proper closed convex function. The same algorithm was also independently proposed and studied by Nesterov \cite{Nesterov2007}. FISTA takes the form \eqref{nest1} and requires $\{\beta_k\}$ to satisfy a certain recurrence relation.
It was shown in \cite{Beck2009,Nesterov2007} that this algorithm exhibits a faster convergence rate than the proximal gradient algorithm, which is
\[
F(x^k) - \inf_{x\in \R^n} F = O\left(\frac1{k^2}\right),
\]
where $\{x^k\}$ is generated by FISTA. Many accelerated proximal gradient algorithms based on Nesterov's extrapolation techniques have been proposed since then, and we refer the readers to \cite{Becker2011,BCG11,Tseng2010} and the references therein for an overview of these algorithms.

The faster convergence rate of FISTA in terms of objective values motivates subsequent studies on the extrapolation scheme \eqref{nest1}; see, for example, \cite{Attouch2015,Chambolle2014,Donoghue2013,Johnstone2015,Zhang2015}.
Particularly, O'Donoghue and Cand\`{e}s \cite{Donoghue2013} proposed an adaptive restart scheme for $\beta_k$ based on FISTA for solving \eqref{compro} with $f$ being convex and $g = 0$. Specifically, instead of following the recurrence relation of $\beta_k$ in FISTA for all $k$, they reset $\beta_k = \beta_0$ every $K$ iterations, where $K$ is a positive number. They established {\em global} linear convergence of the function values when $f$ is strongly convex if $K$ is sufficiently large. Their algorithm is robust against errors in the estimation of the strong convexity modulus of $f$; see the discussion in \cite[Section~2.1]{Donoghue2013}. Later, Attouch and Chbani \cite{Attouch2015}, and independently, Chambolle and Dossal \cite{Chambolle2014}, established the convergence of the whole sequence generated by \eqref{nest1} for solving \eqref{compro} when $f$ is convex and $\beta_k= \frac{k-1}{k+\alpha-1}$ for any fixed $\alpha>3$. More recently, Tao, Boley and Zhang \cite{Zhang2015} established local linear convergence of FISTA applied to the LASSO (i.e., $f$ is a least squares loss function and $g$ is a positive multiple of the $\ell_1$ norm) under the assumption that the problem has a unique solution that satisfies strict complementarity. Johnstone and Moulin \cite{Johnstone2015} considered \eqref{compro} with $f$ being convex, and showed that the whole sequence generated by \eqref{nest1} is convergent by assuming that the extrapolation coefficients $\beta_{k}$ satisfy $0\le\beta_k\leq\bar{\beta}$ for some $\bar{\beta}<1$. Moreover, by imposing uniqueness of the optimal solution together with a technical assumption, they showed that the sequence generated by \eqref{nest1} is locally linearly convergent when applied to the LASSO for a particular choice of $\{\beta_k\}$.

Despite the rich literature, we note that the local linear convergence of \eqref{nest1} is only established for a certain type of convex problems with {\em unique} optimal solutions for some specific choices of $\{\beta_k\}$, which can be restrictive for practical applications. Thus, in this paper, we further study the behavior of the sequence $\{x^k\}$ generated by \eqref{nest1}. Specifically, we discuss local linear convergence under more general conditions in the possibly nonconvex case.

In details, under the same error bound condition used in \cite{Luo1993} for analyzing convergence of the proximal gradient algorithm, we show that there is a threshold $\widetilde\beta$ depending on $f$ so that if $\sup_k\beta_k < \widetilde\beta$, then the sequence $\{x^k\}$ generated by \eqref{nest1} converges $R$-linearly to a stationary point of \eqref{compro} and the sequence of the objective value $\{F(x^k)\}$ is also $R$-linearly convergent. In particular, if $f$ is in addition convex, then $\widetilde\beta$ reduces to $1$ and we can conclude that the sequence $\{x^k\}$ generated by FISTA with fixed restart is $R$-linearly convergent to an optimal solution of \eqref{compro}; see Section~\ref{sec3.3}. The error bound condition is satisfied for a wide range of problems including the LASSO, and hence our linear convergence result concerning \eqref{nest1} with a fixed $\mu$ is more general than those discussed in \cite{Johnstone2015}.


The rest of this paper is organized as follows. Section~\ref{sec2} presents some basic notation and preliminary materials. In Section~\ref{sec3}, we establish linear convergence of the iterates generated by the proximal gradient algorithm with extrapolation under the same error bound condition used in \cite{Luo1993}. Linear convergence of the corresponding sequence of function values is also established. FISTA with restart is discussed in Section~\ref{sec3.3}. In Section 4, we perform numerical experiments to illustrate our results.

\section{Notation and preliminaries}\label{sec2}

Throughout this paper, we use $\R^n$ to denote the $n$-dimensional Euclidean space, with its standard inner product denoted by $\langle\cdot,\cdot\rangle$. The Euclidean norm is denoted by $\|\cdot\|$, the $\ell_1$ norm is denoted by $\|\cdot\|_1$ and the $\ell_\infty$ norm is denoted by $\|\cdot\|_\infty$. The vector of all ones is denoted by $e$, whose dimension should be clear from the context. For a matrix $A\in \R^{m\times n}$, we use $A^\top$ to denote its transpose. Finally, for a symmetric matrix $A\in \R^{n\times n}$, we use $\lambda_{\max}(A)$ and $\lambda_{\min}(A)$ to denote its largest and smallest eigenvalue, respectively.

For a nonempty closed set $\mathcal{C} \subseteq \R^n$, its indicator function is defined by
\begin{eqnarray*}
\delta_{\mathcal{C}}(x)=\left\{\begin{array}{ll}
0      &{\rm if} ~~x\in \mathcal{C},\\
+\infty  &{\rm if} ~~x \notin \mathcal{C}.
\end{array}\right.
\end{eqnarray*}
Moreover, we use $\mathrm{dist}(x,\mathcal{C})$ to denote the distance from $x$ to $\mathcal{C}$, where $\mathrm{dist}(x,\mathcal{C})=\inf_{y\in\mathcal{C}}\|x-y\|$. When $\cal C$ is in addition convex, we use ${\rm Proj}_{\cal C}(x)$ to denote the unique closest point on $\mathcal{C}$ to $x$.

The domain of an extended-real-valued function $h: \R^n\to [-\infty,\infty]$ is defined as $\dom h = \left\{x\in \R^{n}:\;  h(x) < +\infty \right\} $. We say that $h$ is proper if it never equals $-\infty$ and $\dom h\neq \emptyset$. Such a function is closed if it is lower semicontinuous.
A proper closed function $h$ is said to be level bounded if the lower level sets of $h$ are bounded, i.e., the set $\left\{x\in\R^{n}:\;  h(x)\leq r \right\}$ is bounded for any $r\in \R$.
For a proper closed convex function $h:{{\mathbb{R}}^{n}}\to \mathbb{R}\cup \{\infty\}$, the subdifferential of $h$ at $x \in \dom h$ is given by
\begin{equation*}
\partial h(x)=\left\{\xi\in{\mathbb{R}}^{n}:\;  h(u)-h(x)-\langle \xi, u-x \rangle \geq 0,\ \forall  u\in {\mathbb{R}}^{n}\right\}.
\end{equation*}
We use $\mathrm{Prox}_{h}(v)$ to denote the proximal operator of a proper closed convex function $h$ at any $v \in {\mathbb{R}}^{n}$, i.e.:
\begin{equation*}
\mathrm{Prox}_{h}(v)= \mathop{\argmin}_{x\in {{\mathbb{R}}^{n}}}\left\{h(x)+\frac{1}{2}\|x-v\|^{2}\right\}.
\end{equation*}
We note that this operator is well defined for any $v\in \R^n$, and we refer the readers to \cite[Chapter~1]{Parikh2013} for properties of the proximal operator.

For an optimal solution $\hat x$ of \eqref{compro}, the following first-order necessary condition always holds, thanks to \cite[Exercise 8.8(c)]{Rockafellar1998}:
\begin{equation}\label{stationary}
0\in \nabla f(\hat{x}) + \partial g(\hat{x}),
\end{equation}
where $\nabla f$ denotes the gradient of $f$.
We say that $\tilde{x}$ is a stationary point of \eqref{compro} if $\tilde{x}$ satisfies \eqref{stationary} in place of $\hat{x}$; in particular, any optimal solution $\hat x$ of \eqref{compro} is a stationary point of \eqref{compro}. We use $\mathcal{X}$ to denote the set of stationary points of $F$.

Finally, we recall two notions of (local) linear convergence, which will be used in our convergence analysis.
For a sequence $\left\{x^{k}\right\}$, we say that $\left\{x^{k}\right\}$ converges $Q$-linearly to $x^{*}$ if there exist $c \in (0,1)$ and $k_0 > 0$ such that
\begin{equation*}
\|x^{k+1}-x^{*}\|\le c\|x^{k}-x^{*}\|, ~~~\forall k\ge k_0;
\end{equation*}
and we say that $\left\{x^{k}\right\}$ converges $R$-linearly to $x^{*}$ if
\begin{equation*}
\begin{split}
&\limsup\limits_{k\rightarrow \infty}\|x^{k}-x^{*}\|^{\frac{1}{k}}<1.
\end{split}
\end{equation*}
We state the following simple fact relating the two notions of linear convergence, which is an immediate consequence of the definitions of $Q$- and $R$-linear convergence. We will use this fact in our convergence analysis.

\begin{lemma}\label{prelim}
Suppose that $\{a_{k}\}$ and $\{b_{k}\}$ are two sequences in $\R$ with $0\le b_k \le a_k$ for all $k$, and $\{a_{k}\}$ is Q-linearly convergent to zero. Then $\{b_k\}$ is R-linearly convergent to zero.
\end{lemma}

\section{Convergence analysis of the proximal gradient algorithm with extrapolation}\label{sec3}

In this section, we present the proximal gradient algorithm with extrapolation for solving \eqref{compro}, and discuss the convergence behavior of the sequence generated by the algorithm.

We recall that in our problem \eqref{compro}, the function $g$ is proper closed convex and $f$ has a Lipschitz continuous gradient; moreover, $\inf F > -\infty$ and $\mathcal{X}\neq \emptyset$. Furthermore, we observe that
any function $f$ whose gradient is Lipschitz continuous can be written as $f= f_1-f_2$, where $f_{1}$ and $f_{2}$ are two convex functions with Lipschitz continuous gradients. For instance, one can decompose $f$ as
\[
f(x) = \underbrace{f(x) + \frac{c}2 \|x\|^2}_{f_1(x)} - \underbrace{\frac{c}2 \|x\|^2}_{f_2(x)},
\]
for any $c \ge L_f$, where $L_f$ is a Lipschitz continuity modulus of $\nabla f$. It is then routine to show that both $f_1$ and $f_2$ are convex functions with Lipschitz continuous gradients.

Thus, without loss of generality, from now on, we assume that $f = f_1 - f_2$ for some convex functions $f_1$ and $f_2$ with Lipschitz continuous gradients. For concreteness, we denote a Lipschitz continuity modulus of $\nabla f_1$ by $L > 0$, and a Lipschitz continuity modulus of $\nabla f_2$ by $l \ge 0$. Moreover, by taking a larger $L$ if necessary, we assume throughout that $L\ge l$. Then it is not hard to show that $\nabla f$ is Lipschitz continuous with a modulus $L$.

We are now ready to present our proximal gradient algorithm with extrapolation.

\begin{center}
\fbox{\parbox{5in}{\vspace{1mm}
~\textbf{Algorithm 1}: Proximal gradient algorithm with extrapolation
\begin{description}
\item{\textbf{Input}:} $x^{0}\in \dom g$, $\{\beta_{k}\}\subseteq \left[0,\sqrt{\frac{L}{L+l}}\,\right]$. Set $x^{-1}=x^{0}$.

\item \hspace{5mm} ~\textbf{for} $k=0,1,2,\cdots$ ~\textbf{do}\vspace{-1mm}
\begin{equation}\label{iteratePG}
\begin{split}
&y^{k}=x^{k}+\beta_{k}(x^{k}-x^{k-1}),\\
&x^{k+1}={\rm Prox}_{\frac1L g}\left(y^k - \frac1L \nabla f(y^k)\right).\\
\end{split}
\end{equation}
\item \vspace{0mm}\hspace{5mm} \textbf{end for}
\end{description}}}
\end{center}

We shall discuss the convergence behavior of Algorithm 1. We note first that it is immedidate from the definition of the proximal operator that the $x$-update in \eqref{iteratePG} is equivalently given by
\begin{equation}\label{iterateProx}
  x^{k+1} = \argmin_{x\in {\R^n}}\left\{\langle\nabla f(y^k),x\rangle + \frac{L}{2}\|x - y^k\|^2 + g(x)\right\}.
\end{equation}
This fact will be used repeatedly in our convergence analysis below. Our analysis also relies heavily on the following auxiliary sequence:
\begin{equation}\label{Hdefinition}
H_{k,\alpha}=F(x^{k})+\alpha\|x^{k}-x^{k-1}\|^{2},
\end{equation}
for a fixed $\alpha\in[\frac{L+l}{2}\bar\beta^{2},\frac{L}{2}]$ with $\bar\beta := \sup_k \beta_k$, where $\{x^k\}$ is generated by Algorithm 1. We study the convergence properties of $\{H_{k,\alpha}\}$ in Section~\ref{sec3.1}. The results will then be used in subsequent subsections for analyzing the convergence of $\{x^k\}$ and $\{F(x^k)\}$. The auxiliary sequence \eqref{Hdefinition} was also used in \cite{Attouch2015,Chambolle2014,Johnstone2015} for analyzing \eqref{nest1}.

\subsection{Auxiliary lemmas}\label{sec3.1}
We start by showing that $\{H_{k,\alpha}\}$ is nonincreasing and convergent.
\begin{lemma}\label{sucgo0}Let $\{x^{k}\}$ be a sequence generated by Algorithm 1 and $\alpha\in [\frac{L+l}{2}\bar\beta^{2},\frac{L}{2}]$. Then the following statements hold.
\begin{enumerate}[{\rm (i)}]
  \item For any $z\in \dom g$, we have
  \begin{equation}\label{rel1}
  F(x^{k+1}) \le F(z) + \frac{L+l}2\|z - y^k\|^2 - \frac{L}2\|x^{k+1}-z\|^2.
  \end{equation}
  \item It holds that for all $k$,
  \begin{equation}\label{Hdiffer}
H_{k+1,\alpha}-H_{k,\alpha}\le \left(-\frac{L}{2}+\alpha\right)\|x^{k+1}-x^{k}\|^{2}+\left(\frac{L+l}{2}\beta_{k}^{2}-\alpha\right)\|x^{k}-x^{k-1}\|^{2}.
  \end{equation}
  \item The sequence $\{H_{k,\alpha}\}$ is nonincreasing.
\end{enumerate}
\end{lemma}
\begin{proof}
We first prove (i). Fix any $z\in \dom g$. Using the definition of $x^{k+1}$ in \eqref{iterateProx} and the strong convexity of the objective in the minimization problem \eqref{iterateProx}, we obtain upon rearranging terms that
\begin{equation}\label{AproxP}
\begin{split}
g(x^{k+1})\leq&\ g(z)+\langle-\nabla f(y^{k}), x^{k+1}-z\rangle+\frac{L}{2}\|z-y^{k}\|^{2}\\
&-\frac{L}{2}\|x^{k+1}-y^{k}\|^{2}-\frac{L}{2}\|x^{k+1}-z\|^{2}.
\end{split}
\end{equation}
On the other hand, using the fact that $\nabla f$ is Lipschitz continuous with a Lipschitz continuity modulus $L$, we have
\begin{equation}\label{fLip}
f(x^{k+1})\leq f(y^k)+\langle\nabla f(y^k), x^{k+1}-y^k\rangle+\frac{L}{2}\|x^{k+1}-y^k\|^{2}.
\end{equation}
Summing \eqref{AproxP} and \eqref{fLip}, we see further that
\begin{equation}\label{sumfg}
\begin{split}
f(x^{k+1})+g(x^{k+1})\leq&\ f(y^{k})+g(z)+\langle\nabla f(y^{k}), z-y^{k}\rangle\\
&+\frac{L}{2}\|z-y^{k}\|^{2}-\frac{L}{2}\|x^{k+1}-z\|^{2}.
\end{split}
\end{equation}
Next, recall that $f=f_{1}-f_{2}$. Hence, we have
\begin{equation}\label{fcompose}
\begin{split}
&f(y^{k})+\langle\nabla f(y^{k}), z-y^{k}\rangle\\
&=f_{1}(y^{k})-f_{2}(y^{k})+\langle\nabla f_{1}(y^{k}), z-y^{k}\rangle-\langle\nabla f_{2}(y^{k}), z-y^{k}\rangle.
\end{split}
\end{equation}
Since $f_{1}$ and $f_{2}$ are convex and their gradients are Lipschitz continuous with moduli $L$ and $l$, respectively, the following two inequalities hold.
\begin{equation*}
\begin{split}
f_{1}(y^{k})+\langle\nabla f_{1}(y^{k}), z-y^{k}\rangle&\leq f_{1}(z), \\
f_{2}(z)-f_{2}(y^{k})-\langle\nabla f_{2}(y^{k}), z-y^{k}\rangle&\leq\frac{l}{2}\|z-y^{k}\|^{2}.
\end{split}
\end{equation*}
Combining these relations with \eqref{fcompose} and recalling that $f = f_1 - f_2$, we see further that
\begin{equation}\label{fcompose1}
f(y^{k})+\langle\nabla f(y^{k}), z-y^{k}\rangle\leq f(z)+\frac{l}{2}\|z-y^{k}\|^{2}.
\end{equation}
Summing \reff{sumfg} and \reff{fcompose1}, and recalling that $F = f + g$, we obtain \eqref{rel1} immediately. This proves (i).

We now prove (ii). We note first from the definition of the $y$-update in \eqref{iteratePG} that $y^k - x^k = \beta_k(x^k - x^{k-1})$. Using this and \eqref{rel1} with $z = x^k$, we obtain that
\begin{equation*}
F(x^{k+1})-F(x^{k}) \leq \frac{L+l}{2}\beta_{k}^{2}\|x^{k}-x^{k-1}\|^{2}-\frac{L}{2}\|x^{k+1}-x^{k}\|^{2}.
\end{equation*}
From this and the definition of $H_{k,\alpha}$ from \eqref{Hdefinition}, we see further that
\begin{equation*}
\begin{split}
&H_{k+1,\alpha}-H_{k,\alpha}=F(x^{k+1})+\alpha\|x^{k+1}-x^{k}\|^{2}-F(x^{k})-\alpha\|x^{k}-x^{k-1}\|^{2}\\
&\leq -\frac{L}{2}\|x^{k+1}-x^{k}\|^{2}+\frac{L+l}{2}\beta_{k}^{2}\|x^{k}-x^{k-1}\|^{2}+\alpha\|x^{k+1}-x^{k}\|^{2}-\alpha\|x^{k}-x^{k-1}\|^{2}\\
&= \left(-\frac{L}{2}+\alpha\right)\|x^{k+1}-x^{k}\|^{2}+\left(\frac{L+l}{2}\beta_{k}^{2}-\alpha\right)\|x^{k}-x^{k-1}\|^{2},
\end{split}
\end{equation*}
which is just \eqref{Hdiffer}. This proves (ii). Finally, since $\frac{L+l}{2}\bar\beta^{2}\leq\alpha\leq\frac{L}{2}$ by our assumption, we have
\begin{equation*}
-\frac{L}{2}+\alpha \leq 0, \ {\rm and}\  \frac{L+l}{2}\beta_{k}^{2}-\alpha \leq \frac{L+l}{2}\bar\beta^{2} - \alpha \le 0\ \ \forall k.
\end{equation*}
Consequently, $H_{k+1,\alpha}-H_{k,\alpha} \leq 0$, i.e., $\{H_{k,\alpha}\}$ is nonincreasing. This completes the proof.
\end{proof}

The following result is an immediate consequence of Lemma~\ref{sucgo0}.
\begin{corollary}\label{xbound}The sequence $\{x^{k}\}$ generated by Algorithm 1 is bounded if $F$ is level bounded.
\end{corollary}
\begin{proof}
From Lemma \ref{sucgo0}, the sequence $\{H_{k,\frac{L}2}\}$ is nonincreasing. This together with the definition of $H_{k,\frac{L}2}$ implies that
\begin{equation*}
F(x^{k})\leq H_{k,\frac{L}2}\leq H_{0,\frac{L}2}<\infty.
\end{equation*}
Since $F$ is level bounded by assumption, we conclude that $\{x^{k}\}$ is bounded.
\end{proof}

\begin{lemma}\label{ssucgo0}Let $\{x^{k}\}$ be a sequence generated by Algorithm 1, and $\alpha\in [\frac{L+l}{2}\bar\beta^{2},\frac{L}{2}]$. Then the following statements hold.
\begin{enumerate}[{\rm (i)}]
\item The sequence $\{H_{k,\alpha}\}$ is convergent.
\item $\sum_{k=0}^{\infty}\left(\alpha-\frac{L+l}{2}\beta_{k+1}^{2}\right)\|x^{k+1}-x^{k}\|^{2} < \infty$.
\end{enumerate}
\end{lemma}
\begin{proof}
Recall that $\inf F > -\infty$. Hence, $H_{k,\alpha}=F(x^{k})+\alpha\|x^{k}-x^{k-1}\|^{2}$ is bounded from below. This together with the fact that $\{H_{k,\alpha}\}$ is nonincreasing from Lemma~\ref{sucgo0} implies that $\{H_{k,\alpha}\}$ is convergent. This proves (i).

We now prove (ii).
Since $-\frac{L}{2}+\alpha \le 0$, we have from \reff{Hdiffer} that
\begin{equation}\label{less1}
H_{k+1,\alpha}-H_{k,\alpha} \leq -\left(\alpha-\frac{L+l}{2}\beta_{k}^{2}\right)\|x^{k}-x^{k-1}\|^{2}.
\end{equation}
Summing both sides of \eqref{less1} from $1$ to $N$, we see further that
\begin{equation}\label{sumcon}
0\le \sum_{k=1}^{N}\left(\alpha-\frac{L+l}{2}\beta_{k}^{2}\right)\|x^{k}-x^{k-1}\|^{2} \leq \sum_{k=1}^{N}(H_{k,\alpha}-H_{k+1,\alpha})=H_{1,\alpha}-H_{N+1,\alpha},
\end{equation}
where the nonnegativity follows from the fact that $\alpha \ge \frac{L+l}{2}\bar\beta^2\ge \frac{L+l}{2}\beta_k^2$ for all $k$.
Since $\{H_{k,\alpha}\}$ is convergent by (i), letting $N\rightarrow \infty$ in \eqref{sumcon}, we conclude that
the infinite sum exists and is finite, i.e.,
\begin{equation*}
\sum_{k=1}^{\infty}\left(\alpha-\frac{L+l}{2}\beta_{k}^{2}\right)\|x^{k}-x^{k-1}\|^{2} < \infty.
\end{equation*}
This completes the proof.
\end{proof}

In the next lemma, we show that when $\{\beta_k\}$ is chosen below a certain threshold, then any accumulation point of the sequence $\{x^k\}$ generated by Algorithm 1, if exists, is a stationary point of $F$. This result has been established in \cite{Johnstone2015} when the function $f$ is convex. Indeed, in the convex case, it was shown in \cite[Theorem~4.1]{Johnstone2015} that the whole sequence $\{x^k\}$ is convergent. However, the following convergence result is new when the function $f$ is nonconvex.
\begin{lemma}\label{sucgo01}Suppose that $\bar\beta<\sqrt{\frac{L}{L+l}}$ and $\{x^{k}\}$ is a sequence generated by Algorithm 1. Then the following statements hold.
\begin{enumerate}[{\rm (i)}]
\item $\sum_{k=0}^{\infty} \|x^{k+1}-x^{k}\|^{2}<\infty$.
\item Any accumulation point of $\{x^{k}\}$ is a stationary point of $F$.
\end{enumerate}
\end{lemma}
\begin{proof}
Since $\bar\beta<\sqrt{\frac{L}{L+l}}$, one can choose $\alpha \in(\frac{L+l}{2}\bar\beta^{2},\frac{L}{2})$. Then $\frac{L+l}{2}\beta_k^{2}\le\frac{L+l}{2}\bar\beta^{2}<\alpha$ for all $k$, and the conclusion in (i) follows immediately from Lemma~\ref{ssucgo0} (ii).

We next prove (ii). Let $\bar{x}$ be an accumulation point. Then there exists a subsequence $\{x^{k_{i}}\}$ such that $\lim\limits_{i\to\infty}x^{k_{i}}= \bar{x}$. Using the first-order optimality condition of the minimization problem \eqref{iterateProx}, we obtain
\begin{equation*}
-L(x^{k_{i}+1}-y^{k_{i}})\in \nabla f(y^{k_{i}})+\partial g(x^{k_{i}+1}).
\end{equation*}
Combining this with the definition of $y^{k_{i}}$, which is $y^{k_{i}}=x^{k_{i}}+\beta_{k_{i}}(x^{k_{i}}-x^{k_{i}-1})$, we see further that
\begin{equation}\label{Optimali1}
-L[(x^{k_{i}+1}-x^{k_{i}})-\beta_{k_{i}}(x^{k_{i}}-x^{k_{i}-1})]\in \nabla f(y^{k_{i}})+\partial g(x^{k_{i}+1}).
\end{equation}
Passing to the limit in \eqref{Optimali1}, and invoking $\|x^{k_{i}+1}-x^{k_{i}}\|\to 0$ from (i) together with the continuity of $\nabla f$ and the closedness of $\partial g$ (see, for example, \cite[Page~80]{BL06}), we have
\begin{equation*}
0\in \nabla f(\bar{x})+\partial g(\bar{x}),
\end{equation*}
meaning that $\bar{x}$ is a stationary point of $F$. This completes the proof.
\end{proof}

Let $\Omega$ be the set of accumulation points of the sequence $\{x^{k}\}$ generated by Algorithm 1.
Then, from Corollary~\ref{xbound} and Lemma~\ref{sucgo01} $\rm(ii)$, we have $\emptyset \neq \Omega \subseteq\mathcal{X}$ when $F$ is level bounded. We prove in the next proposition that $F$ is constant over $\Omega$ if $\{\beta_k\}$ is chosen below a certain threshold. Since $F$ is only assumed to be lower semicontinuous, this conclusion is nontrivial when $F$ has stationary points that are not globally optimal.

\begin{proposition}\label{Hconstant}Suppose that $\bar\beta<\sqrt{\frac{L}{L+l}}$ and $\{x^{k}\}$ is a sequence generated by Algorithm 1 with its set of accumulation points denoted by $\Omega$. Then $\zeta:=\lim\limits_{k\to\infty}F(x^k)$ exists and $F\equiv \zeta$ on $\Omega$.
\end{proposition}
\begin{proof}
Fix any $\alpha\in (\frac{L+l}{2}\bar\beta^{2},\frac{L}{2})$, which exists because $\bar\beta<\sqrt{\frac{L}{L+l}}$.
Then, in view of Lemmas \ref{ssucgo0} and \ref{sucgo01}, the sequence $\{H_{k,\alpha}\}$ is convergent and $\|x^{k+1}-x^{k}\| \rightarrow 0$. These together with the definition of $H_{k,\alpha}$ imply that $\lim\limits_{k\rightarrow\infty}F(x^{k})$ exists. We denote this limit by $\zeta$.

We now show that $F\equiv\zeta$ on $\Omega$. If $\Omega = \emptyset$, then the conclusion holds trivially. Otherwise, take any $\hat{x}\in \Omega$. Then there exists a convergent subsequence $\{x^{k_{i}}\}$ with $\lim\limits_{i\rightarrow\infty}x^{k_{i}}= \hat{x}$. From the lower semicontinuity of $F$ and the definition of $\zeta$, we have
\begin{equation}\label{lowerinf}
\begin{split}
F(\hat{x}) &\leq \liminf\limits_{i\rightarrow\infty}F(x^{k_{i}})=\zeta.
\end{split}
\end{equation}
On the other hand, using the definition of $x^{k_i}$ as the minimizer in \eqref{iterateProx}, we see that
\begin{equation}\label{uppersup}
g(x^{k_{i}})+\langle \nabla f(y^{k_{i}-1}), x^{k_{i}}-\hat{x}\rangle +\frac{L}{2}\|x^{k_{i}}-y^{k_{i}-1}\|^{2} \leq g(\hat{x})+\frac{L}{2}\|\hat{x}-y^{k_{i}-1}\|^{2}.
\end{equation}
Adding $f(x^{k_{i}})$ to both sides of \eqref{uppersup}, we obtain further that
\begin{equation}\label{uppersup1}
f(x^{k_{i}})+g(x^{k_{i}})+\langle \nabla f(y^{k_{i}-1}), x^{k_{i}}-\hat{x}\rangle +\frac{L}{2}\|x^{k_{i}}-y^{k_{i}-1}\|^{2} \leq f(x^{k_{i}})+g(\hat{x})+\frac{L}{2}\|\hat{x}-y^{k_{i}-1}\|^{2}.
\end{equation}
Next, recall that $y^{k_{i}-1}=x^{k_{i}-1}+\beta_{k_{i}-1}(x^{k_{i}-1}-x^{k_{i}-2})$. Thus, we have
\begin{equation}\label{separate1}
\begin{split}
\|x^{k_{i}}-y^{k_{i}-1}\| &= \|x^{k_{i}}-x^{k_{i}-1}-\beta_{k_{i}-1}(x^{k_{i}-1}-x^{k_{i}-2})\|\\
&\leq \|x^{k_{i}}-x^{k_{i}-1}\|+\bar\beta\|x^{k_{i}-1}-x^{k_{i}-2}\|.
\end{split}
\end{equation}
In addition, we also have
\begin{equation}\label{separate2}
\begin{split}
\|\hat x-y^{k_{i}-1}\| &= \|\hat x-x^{k_{i}}+x^{k_{i}}-y^{k_{i}-1}\|\\
&\leq \|\hat x-x^{k_{i}}\|+\|x^{k_{i}}-y^{k_{i}-1}\|.
\end{split}
\end{equation}
Since $\|x^{k+1}-x^{k}\|\rightarrow 0$ and $\lim\limits_{i\to \infty}x^{k_{i}}=\hat{x}$, it follows from \eqref{separate1} and \eqref{separate2} that
\begin{equation*}
\|x^{k_{i}}-y^{k_{i}-1}\| \rightarrow 0\ {\rm and}\ \|\hat x-y^{k_{i}-1}\|\to 0,
\end{equation*}
and hence $\nabla f(y^{k_{i}-1}) \rightarrow \nabla f(\hat{x})$. From these and \reff{uppersup1}, we obtain that
\begin{equation}\label{upper1}
\zeta =\limsup\limits_{i\rightarrow\infty} F(x^{k_{i}})\leq F(\hat{x}).
\end{equation}
Thus $F(\hat{x})=\lim\limits_{i\rightarrow\infty} F(x^{k_{i}})= \zeta$ from \eqref{lowerinf} and \eqref{upper1}. Since $\hat{x}\in \Omega$ is arbitrary, we see that $F\equiv \zeta$ on $\Omega$. This completes the proof.
\end{proof}

\subsection{Linear convergence of $\{x^{k}\}$ and $\{F(x^k)\}$}\label{sec3.2}
In this subsection, we establish local linear convergence of $\{x^{k}\}$ and $\{F(x^k)\}$ under the following assumption.
\begin{assumption}\label{errorbound}
\begin{enumerate}[{\rm (i)}]
  \item {\bf(Error bound condition)} For any $\xi \geq  \inf_{x\in \R^n} F(x)$, there exist $\epsilon>0$ and $\tau>0$ such that
\[
{\rm dist}(x,\mathcal{X}) \leq \tau \left\|{\rm Prox}_{\frac1L g}\left(x - \frac1L\nabla f(x)\right)-x\right\|,
\]
whenever $\|{\rm Prox}_{\frac1L g}(x - \frac1L\nabla f(x))-x\|<\epsilon$ and $F(x)\leq \xi$.
 \item There exists $\delta>0$, such that $\|x-y\|\geq \delta$ whenever $x,y \in\mathcal{X}$, $F(x)\neq F(y)$.
\end{enumerate}
\end{assumption}
The above assumption has been used in the convergence analysis of many algorithms, including the gradient projection and block coordinate gradient descent method, etc; see, for example, \cite{Beck2006,Luo1992,Luo1993,Tseng1993,Tseng2008,Tseng2009,Tseng2010} and the references therein. The assumption consists of two parts: the first part is an error bound condition, while the second part states that when restricted to $\cal X$, the isocost surfaces of $F$ are properly separated. Under our blanket assumptions on $F$, Assumption~\ref{errorbound} is known to be satisfied for many choices of $f$ and $g$, including:
\begin{itemize}
  \item $f(x)=h(Ax)$, and $g$ is a polyhedral function, where $h$ is twice continuously differentiable on $\R^n$ with a Lipschitz continuous gradient, and on any compact convex set, $h$ is strongly convex; see, \cite[Theorem~2.1]{Luo1992} and \cite[Lemma~6]{Tseng2009}. This covers the well-known LASSO;
  \item $f$ is a possibly nonconvex quadratic function, and $g$ is a polyhedral function; see, for example, \cite[Theorem~4]{Tseng2009}.
\end{itemize}
The first example is convex, while the second one is possibly nonconvex.
We refer the readers to \cite{Tseng2009,Tseng2010,ZZS15} and the references therein for more examples and discussions on the error bound condition.

We next show that $\{H_{k,\alpha}\}$ is $Q$-linearly convergent under Assumption~\ref{errorbound}. Our analysis uses ideas from the proof of \cite[Theorem~2]{Tseng2009}, which studied a block coordinate gradient descent method.
\begin{lemma}\label{HQ1} Suppose that $\bar\beta<\sqrt{\frac{L}{L+l}}$, $\alpha\in (\frac{L+l}{2}\bar\beta^{2},\frac{L}{2})$ and that Assumption \ref{errorbound} holds.
Let $\{x^{k}\}$ be a sequence generated by Algorithm 1. Then the following statements hold.
\begin{enumerate}[{\rm (i)}]
  \item $\lim\limits_{k\rightarrow \infty}\mathrm{dist}(x^{k},\mathcal{X}) = 0$.

  \item The sequence $\{H_{k,\alpha}\}$ is $Q$-linearly convergent.
\end{enumerate}
\end{lemma}

\begin{proof}
First we prove (i). Observe that
\begin{equation}\label{qaddmin}
\begin{split}
&\left\|\mathrm{Prox}_{\frac{1}{L}g}\left(x^{k} - \frac1L \nabla f(x^k)\right)-x^{k}\right\|\\
&\leq \left\|\mathrm{Prox}_{\frac{1}{L}g}\left(x^{k} - \frac1L \nabla f(x^k)\right) - \mathrm{Prox}_{\frac{1}{L}g}\left(y^{k} - \frac1L \nabla f(y^k)\right)\right\|\\
&\ \ +\left\|\mathrm{Prox}_{\frac{1}{L}g}\left(y^{k} - \frac1L \nabla f(y^k)\right)-y^{k}\right\|+\|y^{k}-x^{k}\|.
\end{split}
\end{equation}
We now derive an upper bound for the first term on the right hand side of \eqref{qaddmin}. To this end, using the nonexpansiveness property of the proximal operator (see, for example, \cite[Page~340]{Rockafellar1970}), we have
\begin{equation}\label{project}
\begin{split}
& \left\|\mathrm{Prox}_{\frac{1}{L}g}\left(x^{k}-\frac{1}{L}\nabla f(x^{k})\right)-\mathrm{Prox}_{\frac{1}{L}g}\left(y^{k}-\frac{1}{L}\nabla f(y^{k})\right)\right\|\\
&\leq \left\|x^{k}-\frac{1}{L}\nabla f(x^{k})-y^{k}+\frac{1}{L}\nabla f(y^{k})\right\|\\
& \le \|x^k - y^k\| + \frac{1}{L}\|\nabla f(x^k) - \nabla f(y^k)\| \le 2\|x^k - y^k\|,
\end{split}
\end{equation}
where the last inequality follows from the fact that $\nabla f$ is Lipschitz continuous with modulus $L$.
Combining \eqref{qaddmin}, \eqref{project} and invoking the definition of $x^{k+1}$ in Algorithm 1, we see further that
\begin{equation}\label{contrac}
\begin{split}
& \left\|\mathrm{Prox}_{\frac{1}{L}g}\left(x^{k} - \frac1L \nabla f(x^k)\right)-x^{k}\right\| \leq 3\|x^{k}-y^{k}\|+\|x^{k+1}-y^{k}\|\\
& \le 4\|x^k - y^k\| + \|x^{k+1} - x^k\|\le 4\bar\beta \|x^k - x^{k-1}\| + \|x^{k+1} - x^k\|,
\end{split}
\end{equation}
where the last inequality follows from the definition of $y^k$ in \eqref{iteratePG} and the definition of $\bar\beta$.
Since $\|x^{k+1}-x^{k}\| \to 0$ by Lemma \ref{sucgo01}, we conclude from \eqref{contrac} that
\begin{equation}\label{rel2}
  \left\|\mathrm{Prox}_{\frac{1}{L}g}\left(x^{k} - \frac1L \nabla f(x^k)\right)-x^{k}\right\|\to 0.
\end{equation}
Let $\xi=H_{0,\alpha}$. Since $\{H_{k,\alpha}\}$ is nonincreasing by Lemma~\ref{sucgo0}, we must have $H_{k,\alpha}\leq \xi$ for all $k$, and consequently $F(x^{k})\leq \xi$ for all $k$. In view of this, \eqref{rel2} and Assumption \ref{errorbound} (i), we see that for $\xi = H_{0,\alpha}$, there exist $\tau > 0$ and a positive integer $K$ so that for all $k\ge K$, we have
\begin{equation}\label{dist}
\mathrm{dist}(x^{k},\mathcal{X}) \leq \tau \left\|\mathrm{Prox}_{\frac{1}{L}g}\left(x^{k} - \frac1L \nabla f(x^k)\right)-x^{k}\right\|.
\end{equation}
Thus from \eqref{rel2} and \eqref{dist}, we immediately obtain the conclusion in (i).

We now prove (ii). Take an arbitrary $z\in {\cal X}$, we have from \eqref{rel1} that
\begin{equation}\label{unity}
\begin{split}
F(x^{k+1}) &\leq F(z)+\frac{L+l}{2}\|z-y^{k}\|^{2}-\frac{L}{2}\|x^{k+1}-z\|^{2}\\
& \le F(z)+\frac{L+l}{2}\|z-y^{k}\|^{2}\\
& = F(z) + \frac{L+l}{2}\|z-x^{k}+x^{k}-y^{k}\|^{2}\\
& \le F(z) + (L+l)\|z-x^{k}\|^{2}+(L+l)\|x^{k}-y^{k}\|^{2}.
\end{split}
\end{equation}
Choose $z$ in \eqref{unity} to be an $\bar x^k\in{\cal X}$ so that $\|\bar x^k-x^{k}\| = {\rm dist}(x^k,{\cal X})$. Then we obtain
\begin{equation}\label{Fdistx-1}
F(x^{k+1})-F(\bar x^k)\leq (L+l)\mathrm{dist}^{2}(x^{k},{\cal X})+(L+l)\|x^{k}-y^{k}\|^{2}.
\end{equation}
In addition, recall that $\|x^{k+1} - x^k\|\to 0$ by Lemma \ref{sucgo01}. This together with \eqref{rel2} and \eqref{dist} shows that $\|\bar x^{k+1} - \bar x^k\|\to 0$. In view of this and Assumption~\ref{errorbound} (ii), it must then hold true that $F(\bar x^k)\equiv \zeta$ for some constant $\zeta$ for all sufficiently large $k$. Thus, for all sufficiently large $k$, we have from \eqref{Fdistx-1} that
\begin{equation}\label{Fdistx}
F(x^{k+1})-\zeta\leq (L+l)\mathrm{dist}^{2}(x^{k},{\cal X})+(L+l)\|x^{k}-y^{k}\|^{2}.
\end{equation}
On the other hand, since $\bar x^k$ is a stationary point of \eqref{compro} so that $- \nabla f(\bar x^k) \in \partial g(\bar x^k)$, we have for all $k$ that,
\[
g(\bar x^k) - g(x^k) \le \langle -\nabla f(\bar x^k),\bar x^k - x^k\rangle.
\]
Using this and the definitions of $F$, $H_{k,\alpha}$ and $\zeta$, we see that for all sufficiently large $k$,
\begin{equation*}
\begin{split}
  &\zeta - H_{k,\alpha} = F(\bar x^k) - F(x^k) - \alpha\|x^k - x^{k-1}\|^2 \\
  & = f(\bar x^k) + g(\bar x^k) - f(x^k) - g(x^k) - \alpha\|x^k - x^{k-1}\|^2\\
  & \le f(\bar x^k) - f(x^k) + \langle -\nabla f(\bar x^k),\bar x^k - x^k\rangle - \alpha\|x^k - x^{k-1}\|^2\\
  & = - f(x^k) - [-f(\bar x^k)] - \langle -\nabla f(\bar x^k),x^k - \bar x^k\rangle - \alpha\|x^k - x^{k-1}\|^2\\
  & \le \frac{L}{2}\|x^k - \bar x^k\|^2 - \alpha\|x^k - x^{k-1}\|^2,
\end{split}
\end{equation*}
where the last inequality follows from the Lipschitz continuity of $-\nabla f$. Using this, the fact that $\|x^{k+1} - x^k\|\to 0$ by Lemma~\ref{sucgo01} and the fact that $\|\bar x^k-x^{k}\| = {\rm dist}(x^k,{\cal X}) \rightarrow 0$ by (i), we deduce that
\begin{equation}\label{infH}
\zeta\le \lim_{k\rightarrow \infty}H_{k,\alpha} = \inf_{k}H_{k,\alpha},
\end{equation}
where the equality follows from Lemma \ref{sucgo0} (iii).

Now, from \eqref{contrac}, \eqref{dist} and \eqref{Fdistx}, we see that for all sufficiently large $k$,
\begin{equation*}
\begin{split}
F(x^{k+1})-\zeta&\leq (L+l)\mathrm{dist}^{2}(x^{k},{\cal X})+(L+l)\|x^{k}-y^{k}\|^{2}\\
&\le (L+l)\tau^{2}(4\bar\beta\|x^k - x^{k-1}\| + \|x^{k+1}-x^k\|)^2+(L+l)\|x^{k}-y^{k}\|^{2}\\
&\le (L+l)\tau^2(4\bar\beta\|x^k - x^{k-1}\| + \|x^{k+1}-x^k\|)^2 + (L+l)\bar\beta^2\|x^{k}-x^{k-1}\|^2\\
&\le C(\|x^k - x^{k-1}\|^2 + \|x^{k+1}-x^k\|^2),
\end{split}
\end{equation*}
for some positive constant $C$, where the third inequality follows from the definition of $y^k$ in \eqref{iteratePG} and the definition of $\bar\beta$.
Combining this with the definition of $H_{k,\alpha}$, we obtain further that
\begin{equation}\label{Hbound2}
0\le H_{k+1,\alpha}-\zeta \leq \eta(\|x^{k}-x^{k-1}\|^{2}+\|x^{k+1}-x^{k}\|^{2}),
\end{equation}
where $\eta = C+\alpha$, and the nonnegativity is a consequence of \eqref{infH}.
On the other hand, let $\delta=\min\left\{\frac{L}{2}-\alpha, \alpha-\frac{L+l}{2}\bar\beta^{2}\right\}$. Then $\delta > 0$ and we see from \eqref{Hdiffer} that
\begin{equation}\label{Hdiffer2}
(H_{k+1,\alpha}-\zeta)-(H_{k,\alpha}-\zeta)\leq -\delta(\|x^{k+1}-x^{k}\|^{2}+\|x^{k}-x^{k-1}\|^{2}).
\end{equation}
Combining \eqref{Hdiffer2} and \reff{Hbound2}, we obtain further that
\begin{equation}\label{Hdiffer3}
(H_{k+1,\alpha}-\zeta)-(H_{k,\alpha}-\zeta)\leq -\frac{\delta}{\eta}(H_{k+1,\alpha}-\zeta).
\end{equation}
Reorganizing \reff{Hdiffer3}, we see that for all sufficiently large $k$,
\begin{equation*}
0\le H_{k+1,\alpha}-\zeta \leq \frac{1}{1+\frac{\delta}{\eta}}(H_{k,\alpha}-\zeta),
\end{equation*}
which implies that the sequence $\{H_{k,\alpha}\}$ is $Q$-linearly convergent. This completes the proof.
\end{proof}

We are now ready to prove the local linear convergence of the sequences $\{x^k\}$ and $\{F(x^k)\}$, using the $Q$-linear convergence of $\{H_{k,\alpha}\}$.
\begin{theorem}\label{Linearcon}Suppose that $\bar\beta<\sqrt{\frac{L}{L+l}}$ and that Assumption \ref{errorbound} holds.
Let $\{x^{k}\}$ be a sequence generated by Algorithm 1. Then the following statements hold.
\begin{enumerate}[{\rm (i)}]
\item  The sequence $\{x^{k}\}$ is $R$-linearly convergent to a stationary point of $F$.
\item  The sequence $\{F(x^{k})\}$ is R-linearly convergent.
\end{enumerate}
\end{theorem}
\begin{proof}
Fix any $\alpha \in (\frac{L+l}{2}\bar\beta^{2},\frac{L}{2})$, which exists because $\bar\beta<\sqrt{\frac{L}{L+l}}$.
Then, in view of Lemma~\ref{HQ1}, the sequence $\{H_{k,\alpha}\}$ is $Q$-linearly convergent. For notational simplicity, we denote its limit by $\zeta$. Let $\delta=\min\left\{\frac{L}{2}-\alpha, \alpha-\frac{L+l}{2}\bar\beta^{2}\right\}$. Then $\delta > 0$ and we obtain from \eqref{Hdiffer} that
\begin{equation}\label{eq:newlyadded}
\|x^{k+1}-x^{k}\|^{2} \leq \frac{1}{\delta}(H_{k,\alpha}-\zeta)-\frac{1}{\delta}(H_{k+1,\alpha}-\zeta)\le \frac{1}{\delta}(H_{k,\alpha}-\zeta),
\end{equation}
where the last inequality follows from the fact that the sequence $\{H_{k,\alpha}\}$ is nonincreasing and convergent to $\zeta$, thanks to Lemmas~\ref{sucgo0} and \ref{ssucgo0}. Using the above inequality and the fact that the sequence $\{H_{k,\alpha}\}$ is $Q$-linearly convergent, we see that there exist $0<c<1$ and $M > 0$ such that
\begin{equation}\label{rel3}
\|x^{k+1}-x^k\|\le M c^k
\end{equation}
for all $k$. Consequently, for any $m_2> m_1\ge 1$, we have
\[
\|x^{m_2}- x^{m_1}\| \le \sum_{k=m_1}^{m_2-1}\|x^{k+1}-x^k\|\le \frac{Mc^{m_1}}{1-c},
\]
showing that $\{x^k\}$ is a Cauchy sequence and hence convergent. Denoting its limit by $\hat x$ and passing to the limit as $m_2\to \infty$ in the above relation, we see further that
\[
\|x^{m_1} - \hat x\|\le \frac{Mc^{m_1}}{1-c}.
\]
This means that the sequence $\{x^k\}$ is $R$-linearly convergent to its limit, which is a stationary point of $F$ according to Lemma~\ref{sucgo01}. This proves (i).

Next, we prove (ii). Notice that for any $k\ge 1$, we have from the definition of $H_{k,\alpha}$ that
\begin{equation*}
\begin{split}
  |F(x^k) - \zeta| & = |H_{k,\alpha} - \zeta - \alpha\|x^k - x^{k-1}\|^2| \le H_{k,\alpha} - \zeta + \alpha \|x^k - x^{k-1}\|^2\\
  & \le H_{k,\alpha} - \zeta + \frac{\alpha}{\delta}(H_{k-1,\alpha} - \zeta),
\end{split}
\end{equation*}
where the first inequality follows from the triangle inequality and the fact that the sequence $\{H_{k,\alpha}\}$ is nonincreasing and convergent to $\zeta$ according to Lemmas~\ref{sucgo0} and \ref{ssucgo0},
and the second inequality follows from \eqref{eq:newlyadded}. This together with the $Q$-linear convergence of $\{H_{k,\alpha}\}$ and Lemma~\ref{prelim} implies the $R$-linear convergence of $\{F(x^k)\}$.
\end{proof}

\subsection{FISTA with restart: a special case of Algorithm 1}\label{sec3.3}
In this subsection, we discuss FISTA with restart. Restart schemes for FISTA were proposed recently in O'Donoghue and Cand\`{e}s \cite{Donoghue2013}, 
where they adopted as a heuristic an adaptive restart technique, and established global linear convergence of the objective value when applying their method to \eqref{compro} with $f$ being strongly convex and $g=0$. The restart techniques have also been adopted in the popular software, TFOCS \cite{BCG11}. While they did not prove any linear convergence results for convex nonsmooth problems such as the LASSO, they stated that for the LASSO, ``after a certain number of iterations adaptive restarting can provide linear convergence"; see \cite[Page~728]{Donoghue2013}. In this subsection, we will explain that FISTA equipped with the aforementioned restart schemes is a special case of Algorithm 1. Moreover, when both of their restart schemes are used for the LASSO, both the sequences $\{x^k\}$ and $\{F(x^k)\}$ are $R$-linearly convergent.

To proceed, we first present FISTA \cite{Beck2009,Nesterov2007} for solving \eqref{compro} with $f$ being in addition convex.
\begin{center}
\fbox{\parbox{5in}{\vspace{1mm}
\begin{description}
\item\textbf{FISTA} ~{\textbf{Input}:} $x^{0}\in \dom g$, $\theta_{-1}=\theta_0=1$. Set $x^{-1}=x^{0}$.
\item \hspace{5mm} ~\textbf{for} $k=0,1,2\cdots$ ~\textbf{do}\vspace{-1mm}
\begin{equation*}
\begin{split}
&\beta_{k}=\frac{\theta_{k-1}-1}{\theta_{k}},\\
&y^{k}=x^{k}+\beta_{k}(x^{k}-x^{k-1}),\\
&x^{k+1}=\mathrm{Prox}_{\frac{1}{L}g}\left(y^{k}-\frac{1}{L}\nabla f(y^{k})\right),\\
&\theta_{k+1}= \frac{1+\sqrt{1+4\theta_{k}^{2}}}{2}.
\end{split}
\end{equation*}
\item \vspace{0mm}\hspace{5mm} \textbf{end for}
\end{description}}}
\end{center}

As one of the many variants of Nesterov's accelerated proximal gradient algorithms, FISTA uses a specific choice of $\{\beta_{k}\}$. According to the formula for updating $\beta_{k}$ in FISTA above, it holds that $0\le\beta_{k}<1$ for all $k$.\footnote{Since $\theta_{k+1} = \left(1+\sqrt{1+4\theta_{k}^{2}}\right)/2$ and $\theta_{-1}=\theta_0=1$ in FISTA, by induction, it is routine to show that $\theta_k \ge \frac32$ and $\theta_{k-1}-1 <\theta_k$ whenever $k\ge 1$. Combining these with the definition of $\beta_k$ in FISTA, we see that $0\le\beta_{k}<1$ for all $k$.} On the other hand, since $f$ is convex, we can choose $l = 0$ and thus $\sqrt{\frac{L}{L+l}} = 1$ in Algorithm 1. Consequently, FISTA can be viewed as a special case of Algorithm 1.

FISTA with restart (see, for example, \cite{BCG11,Donoghue2013}) is based on FISTA. Here, we adopt the same restart schemes as in \cite{Donoghue2013}: fixed restart and adaptive restart. In the fixed restart scheme, we choose a positive integer $K$ and reset $\theta_{k-1} = \theta_k=1$ every $K$ iterations, while in the adaptive restart (gradient scheme),\footnote{There is also another scheme based on function values. It was discussed in \cite[Section~3.2]{Donoghue2013} that the two schemes perform similarly empirically and that the gradient scheme has advantages over the function value scheme. Thus, in this paper, we focus on the gradient scheme.} we reset $\theta_k = \theta_{k+1} = 1$ whenever $\langle y^{k}-x^{k+1},x^{k+1}-x^{k}\rangle >0$; see \cite[Eq.~13]{Donoghue2013}. Clearly, whenever the fixed restart scheme is invoked, we will have $\bar\beta<1$. Thus, we have the following immediate corollary of Theorem~\ref{Linearcon}.

\begin{corollary}\label{propFISTA} Suppose that $f$ in \eqref{compro} is convex and Assumption~\ref{errorbound} holds. Let $\{x^k\}$ be a sequence generated by FISTA with the fixed restart scheme or both the fixed and adaptive restart schemes. Then
  \begin{enumerate}[{\rm (i)}]
    \item $\{x^k\}$ converges $R$-linearly to a globally optimal solution of \eqref{compro}.
    \item $\{F(x^k)\}$ converges $R$-linearly to the globally optimal value of \eqref{compro}.
  \end{enumerate}
\end{corollary}

From the discussion following Assumption~\ref{errorbound}, we see that the objective function in the LASSO satisfies Assumption~\ref{errorbound}. Thus, by Corollary~\ref{propFISTA}, when the fixed or both the fixed and adaptive restart schemes are used for the LASSO, both the sequences $\{x^k\}$ and $\{F(x^k)\}$ are $R$-linearly convergent.

Before ending this subsection, we would like to point out two crucial differences between our Corollary~\ref{propFISTA} and the conclusion in \cite{Donoghue2013}. First, they concluded {\em global} linear convergence of function values for a special case of \eqref{compro} where $f$ is strongly convex and $g = 0$, while we obtain {\em local} linear convergence for \eqref{compro} for both $\{x^k\}$ and $\{F(x^k)\}$ with $f$ being convex. Second, their global linear convergence is only guaranteed if $K$ is chosen sufficiently large; see \cite[Eq.~6]{Donoghue2013}. On the other hand, we do not have any restrictions on the number $K$, the width of the restart interval.

\section{Numerical experiments}\label{sec4}

In this section, we conduct numerical experiments to study Algorithm 1 under different choices of $\{\beta_k\}$. We consider three different types of problems: the $\ell_1$ regularized logistic regression problem, the LASSO, and the problem of minimizing a nonconvex quadratic function over a simplex. The first two problems are convex optimization problems, while the third problem is possibly nonconvex. We consider three different algorithms for each class of problems. For the convex problems, we consider Algorithm 1 with $\beta_k \equiv 0$ (proximal gradient algorithm), $\beta_k$ chosen as in FISTA, and $\beta_k$ chosen as in FISTA with both the fixed and the adaptive restart schemes. On the other hand, for the nonconvex problems, we consider Algorithm 1 with $\beta_k \equiv 0$ (proximal gradient algorithm) and $\beta_k \equiv 0.98\sqrt{\frac{L}{L + l}}$. We also consider  FISTA as a heuristic.

All the numerical experiments are performed in Matlab 2014b on a 64-bit PC with an Intel(R) Core(TM) i7-4790 CPU (3.60GHz) and 32GB of RAM.

\subsection{$\ell_1$ regularized logistic regression}
In this subsection, we consider the $\ell_1$ regularized logistic regression problem:
\begin{equation}\label{logistic}
v_{_{\rm log}}:= \mathop{\min}_{\tilde{x}\in {{\mathbb{R}}^{n}},x_0\in \R} \sum_{i=1}^m \log(1+\exp(-b_i(a^\top_i\tilde{x}+x_0)))+\lambda \|\tilde{x}\|_1,
\end{equation}
where $a_i\in \R^n$, $b_i \in \{-1, 1\}$, $i=1,2,\cdots, m$, with $b_i$ not all the same, $m<n$ and $\lambda >0$ is the regularization parameter. It is easy to see that \eqref{logistic} is in the form of \eqref{compro} with
\begin{equation}\label{logisticfg}
f(x)=\sum_{i=1}^m \log(1+\exp(-b_i(Dx)_i)), ~~~~g(x)=\lambda\|\tilde{x}\|_1,
\end{equation}
where $x:= (\tilde x,x_0)\in \R^{n+1}$, and $D$ is the matrix whose $i$th row is given by $(a^\top_i ~~1)$. Moreover, one can show that $\nabla f$ is Lipschitz continuous with modulus $0.25\lambda_{\max}(D^\top D)$. Thus, in our algorithms below, we take $L = 0.25\lambda_{\max}(D^\top D)$ and $l = 0$.

Before applying Algorithm 1, we need to show that $v_{_{\rm log}} > -\infty$ and the solution set ${\cal X}$ of \eqref{logistic} is nonempty. To this end, we first recall that
%
the dual problem of \eqref{logistic} is given by
\begin{equation}\label{dual_log}
\begin{array}{rl}
\max\limits_{u \in \R^m}&d_{_{\rm log}}(u) := -\sum_{i=1}^m [-b_i u_i\log(-b_i u_i)+(1+b_i u_i)\log(1+b_i u_i)]\\
\mathrm{s.t.}& \|A^\top u\|_\infty \le \lambda, ~~ e^\top u=0,
\end{array}
\end{equation}
where $A$ is the matrix whose $i$th row is $a^\top_i$. It can be shown that the optimal values of \eqref{logistic} and \eqref{dual_log} are the same, and that an optimal solution of \eqref{dual_log} exists; see, for example, \cite[Theorem~3.3.5]{BL06}. In addition, we note that because $\lambda>0$ and $b_i$ are not all the same, the generalized Slater condition is satisfied for \eqref{dual_log}, i.e., there exists $\tilde u$ satisfying $\|A^\top\tilde u\|_\infty < \lambda$, $e^\top \tilde u = 0$ and $-1<b_i\tilde u_i < 0$ for $i=1,\ldots,m$. Hence, by \cite[Corollary 28.2.2]{Rockafellar1970}, an optimal solution of \eqref{logistic} exists. Consequently, $v_{_{\rm log}} > -\infty$ and the solution set ${\cal X}$ of \eqref{logistic} is nonempty.

 Thus, Algorithm 1 is applicable. In addition, from the discussion following Assumption~\ref{errorbound}, Assumption~\ref{errorbound} is satisfied for \eqref{logisticfg}. Hence, one should expect $R$-linear convergence of the sequences $\{x^k\}$ and $\{F(x^k)\}$ generated by FISTA with restart, in view of Corollary~\ref{propFISTA}.

We now perform numerical experiments to study Algorithm 1 under three choices of $\{\beta_k\}$: $\beta_k \equiv 0$ as in the proximal gradient algorithm (PG), $\beta_k$ chosen as in FISTA, and $\beta_k$ chosen as in FISTA with both the fixed and the adaptive restart schemes, where we perform a fixed restart every $500$ iterations (FISTA-R500). We choose $\lambda = 5$ in \eqref{logistic} and initialize all three algorithms at the origin. As for the termination, we make use of the fact that for any $\bar x\in {\cal X}$, $\nabla p(D\bar x)$ is an optimal solution of \eqref{dual_log} (see, for example, \cite[Theorem~31.3]{Rockafellar1970}). Specifically, we define
\begin{equation*}
u^k=\min\left\{1,\frac{\lambda}{\left\|A^\top \nabla p(Dx^k)\right\|_\infty}\right\}\nabla p(Dx^k),
\end{equation*}
and terminate the algorithms once the duality gap and the dual feasibility violation are small, i.e.,
\begin{equation*}
\max\left\{\frac{|f(x^k) + g(x^k) - d_{_{\rm log}}(u^k)|}{\max\{f(x^k) + g(x^k),1\}}, \frac{50|e^Tu^k|}{\max\{\|u^k\|,1\}}\right\} \le 10^{-6}.
\end{equation*}
We also terminate the algorithms when the number of iterations hits $5000$.

We consider random instances for our experiments. For each $(m,n,s) = (300,3000,30)$, $(500,5000,50)$ and $(800,8000,80)$, we generate an $m\times n$ matrix $A$ with i.i.d. standard Gaussian entries. We then choose a support set $T$ of size $s$ uniformly at random, and generate an $s$-sparse vector $\hat x$ supported on $T$ with i.i.d. standard Gaussian entries. The vector $b$ is then generated as $b = {\rm sign}(A\hat x + ce)$, where $c$ is chosen uniformly at random from $[0,1]$.

Our computational results are presented in Figures~\ref{log3000}, \ref{log5000} and \ref{log8000}. In the plot (a) of each figure, we plot $\|x^k-x^*\|$ against the number of iterations, where $x^*$ denotes the approximate solution obtained at termination of the respective algorithm; while in the plot (b) of each figure, we plot $|F(x^k) - F_{\rm min}|$ against the number of iterations, where $F_{\rm min}$ denotes the minimum of the three objective values obtained from the three algorithms. We see that both $\{x^k\}$ and $\{F(x^k)\}$ generated by FISTA with both fixed and adaptive restart schemes are $R$-linearly convergent, which conforms with our theory. Moreover, compared with FISTA and the proximal gradient algorithm, the algorithm with restart performs better.

\begin{figure}[h]
\centering
\caption{$l_1-logistic:~n=3000, m=300, s=30$}\label{log3000}
\subfigure[]{\includegraphics[height=3.8cm]{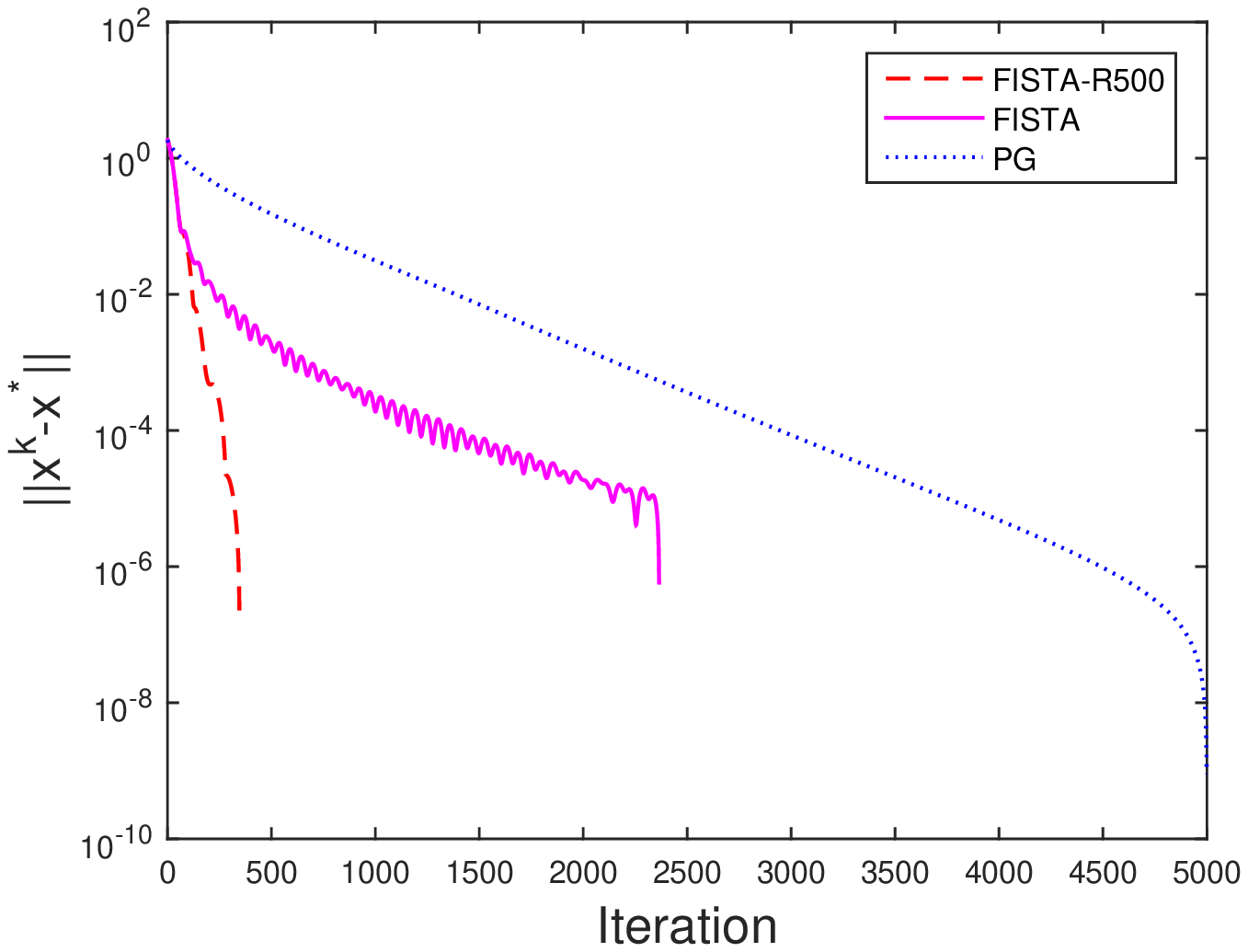}}
\subfigure[]{\includegraphics[height=3.8cm]{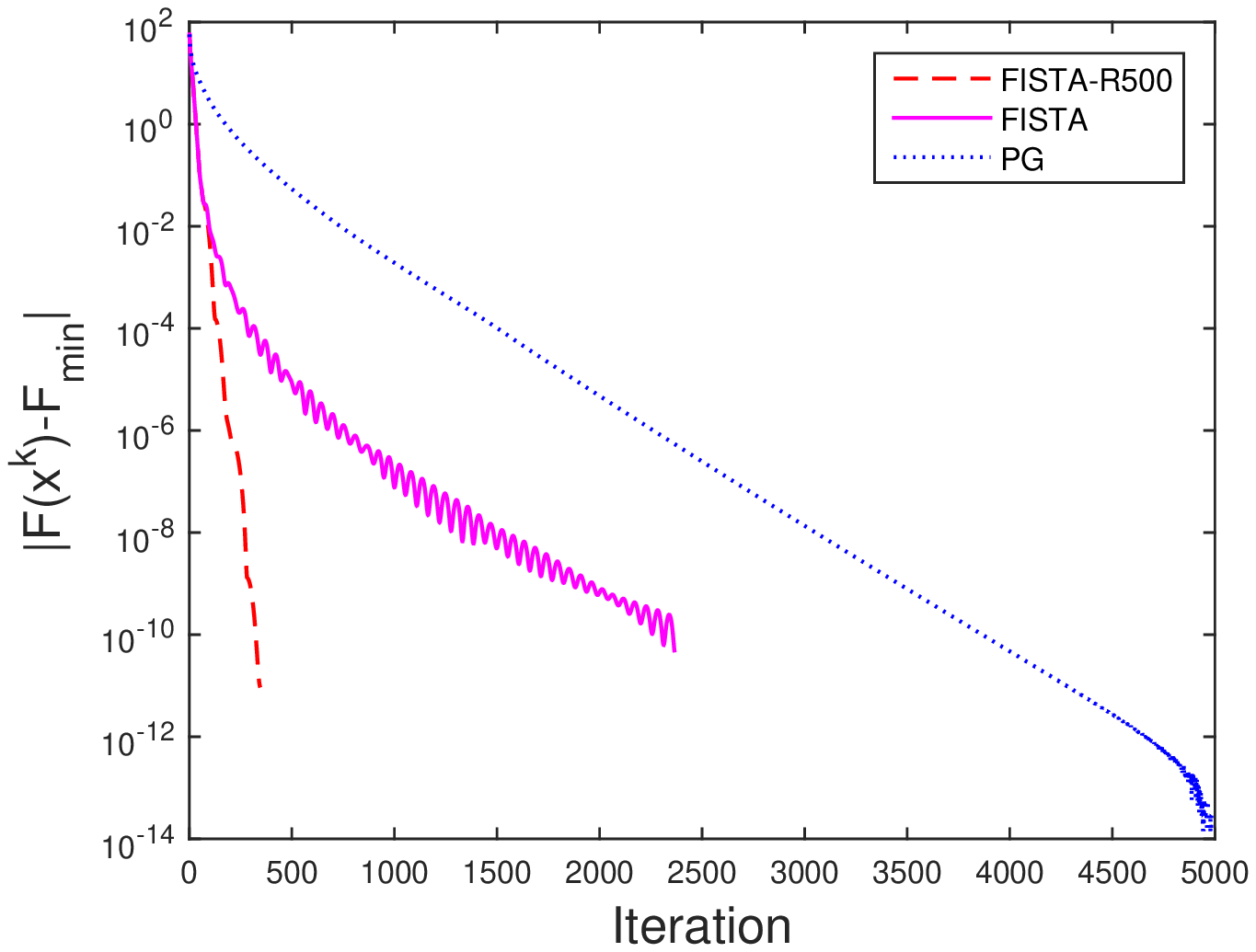}}
\end{figure}

\begin{figure}[h]
\centering
\caption{$l_1-logistic:~n=5000, m=500, s=50$}\label{log5000}
\subfigure[]{\includegraphics[height=3.8cm]{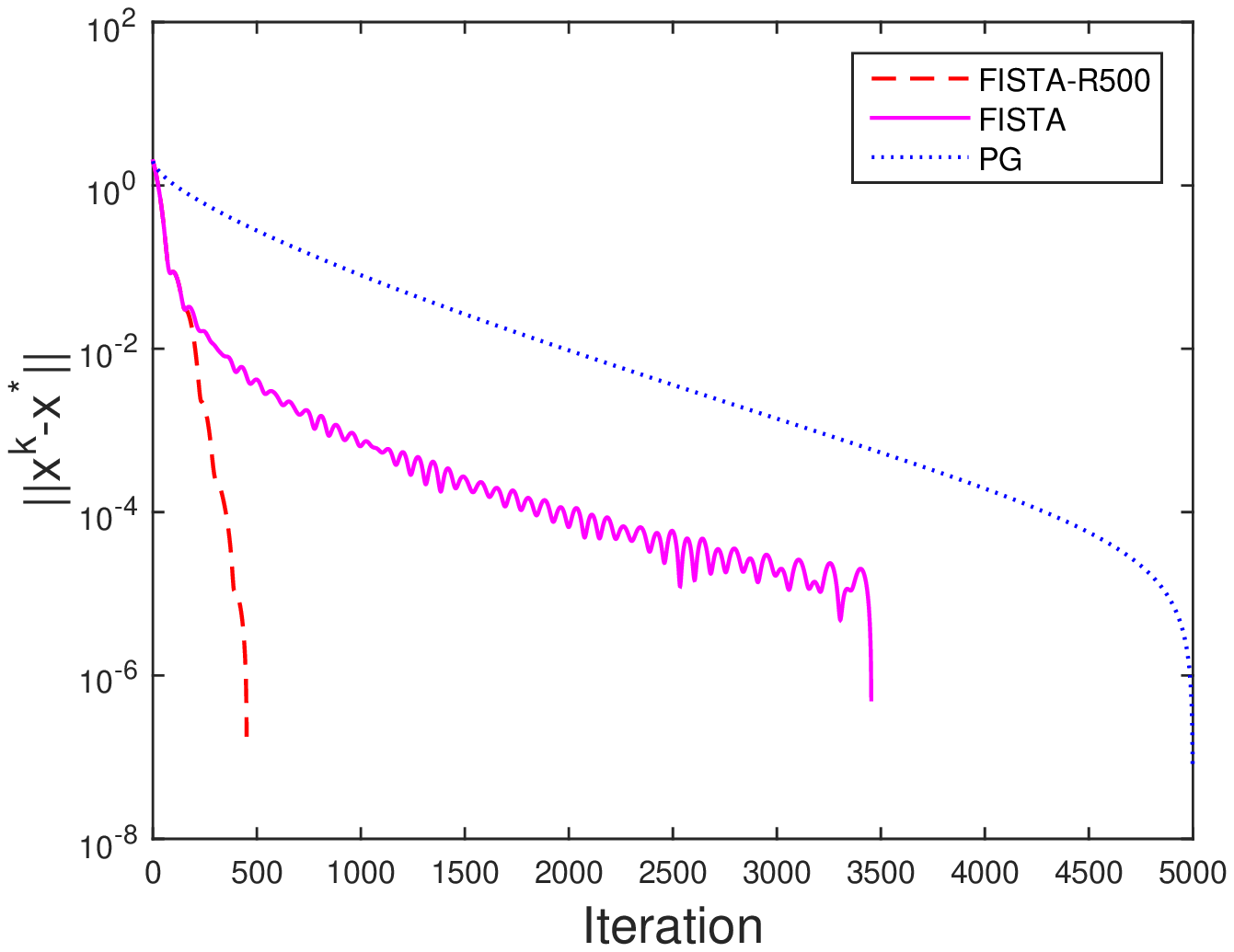}}
\subfigure[]{\includegraphics[height=3.8cm]{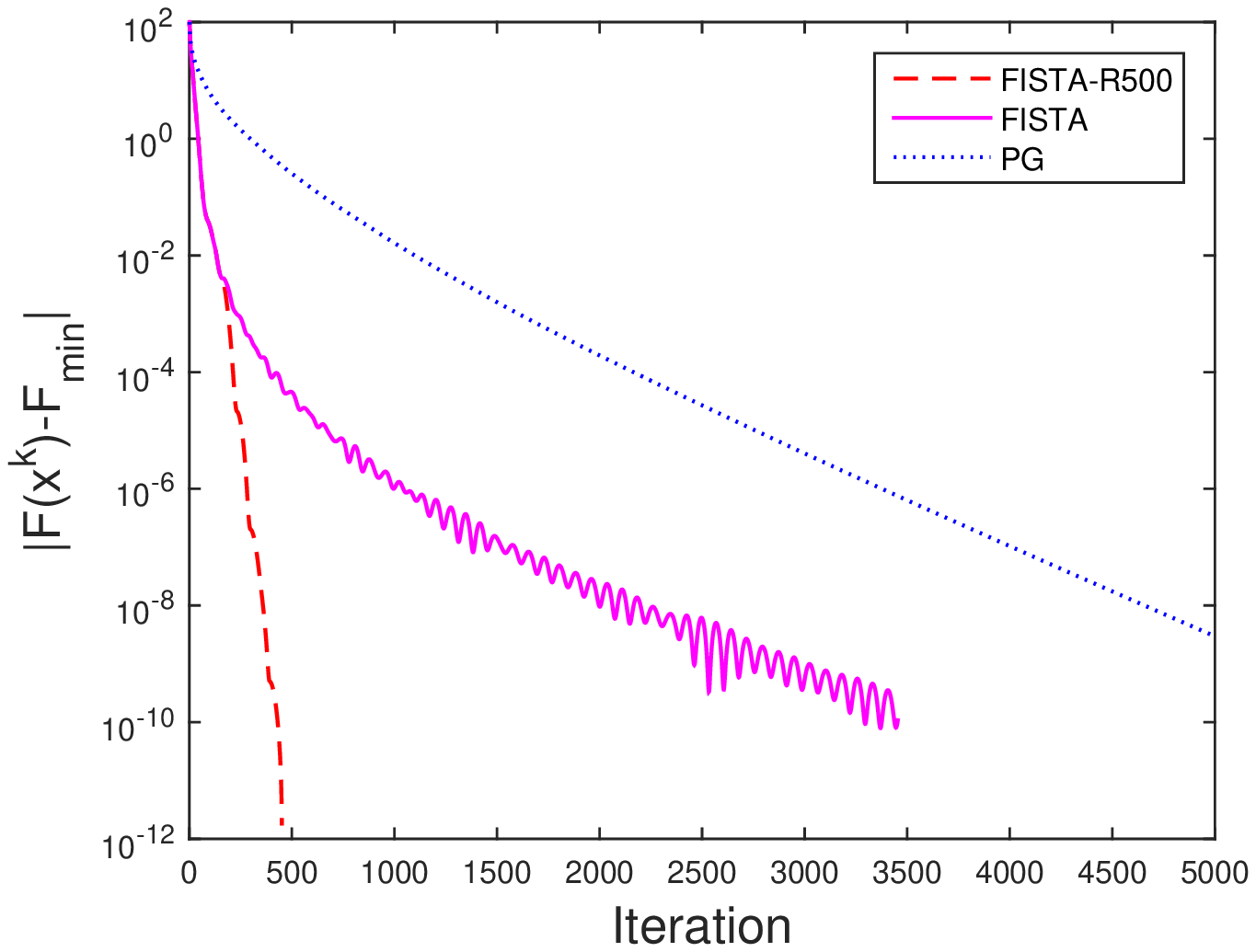}}
\end{figure}

\begin{figure}[h]
\centering
\caption{$l_1-logistic:~n=8000, m=800, s=80$}\label{log8000}
\subfigure[]{\includegraphics[height=3.8cm]{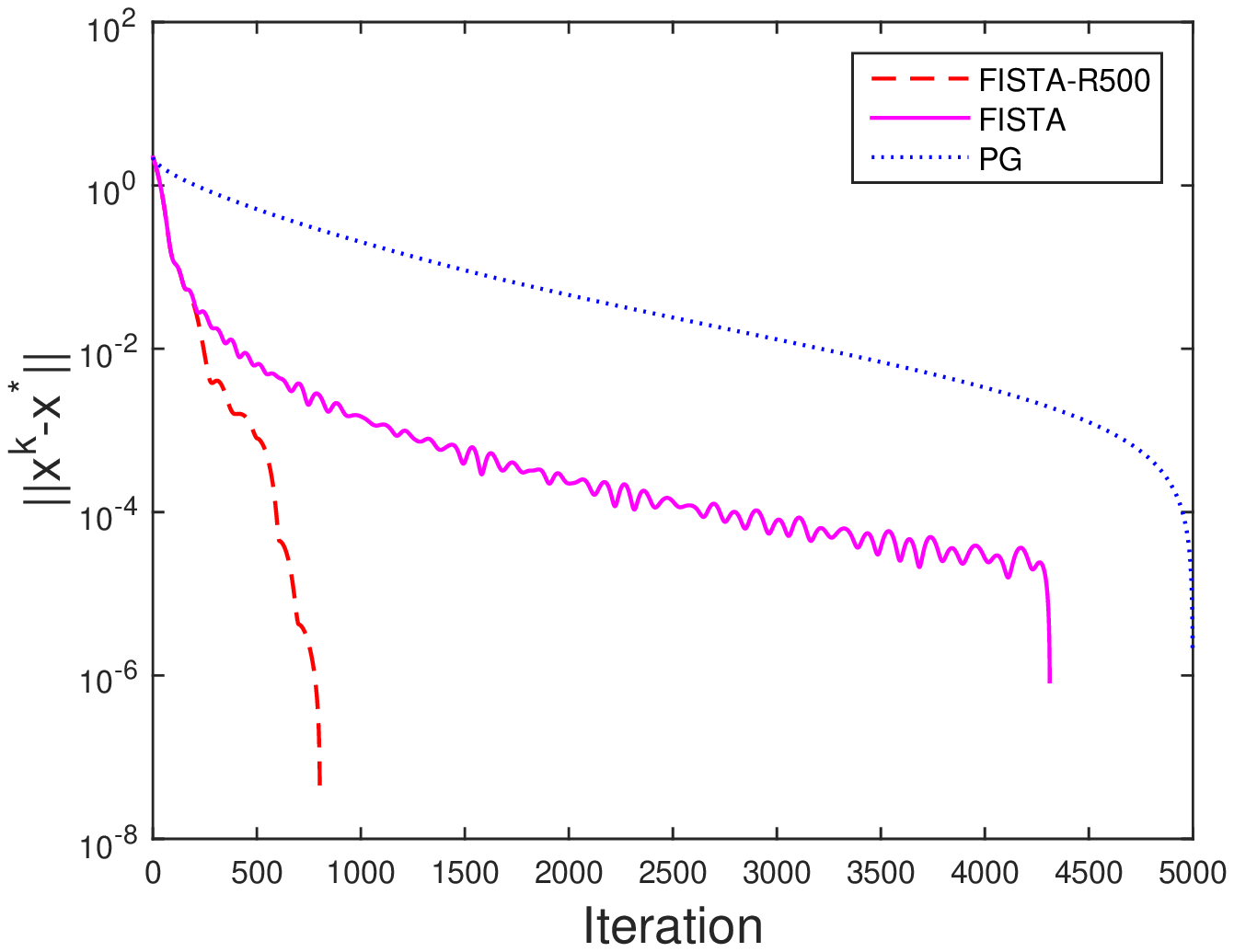}}
\subfigure[]{\includegraphics[height=3.8cm]{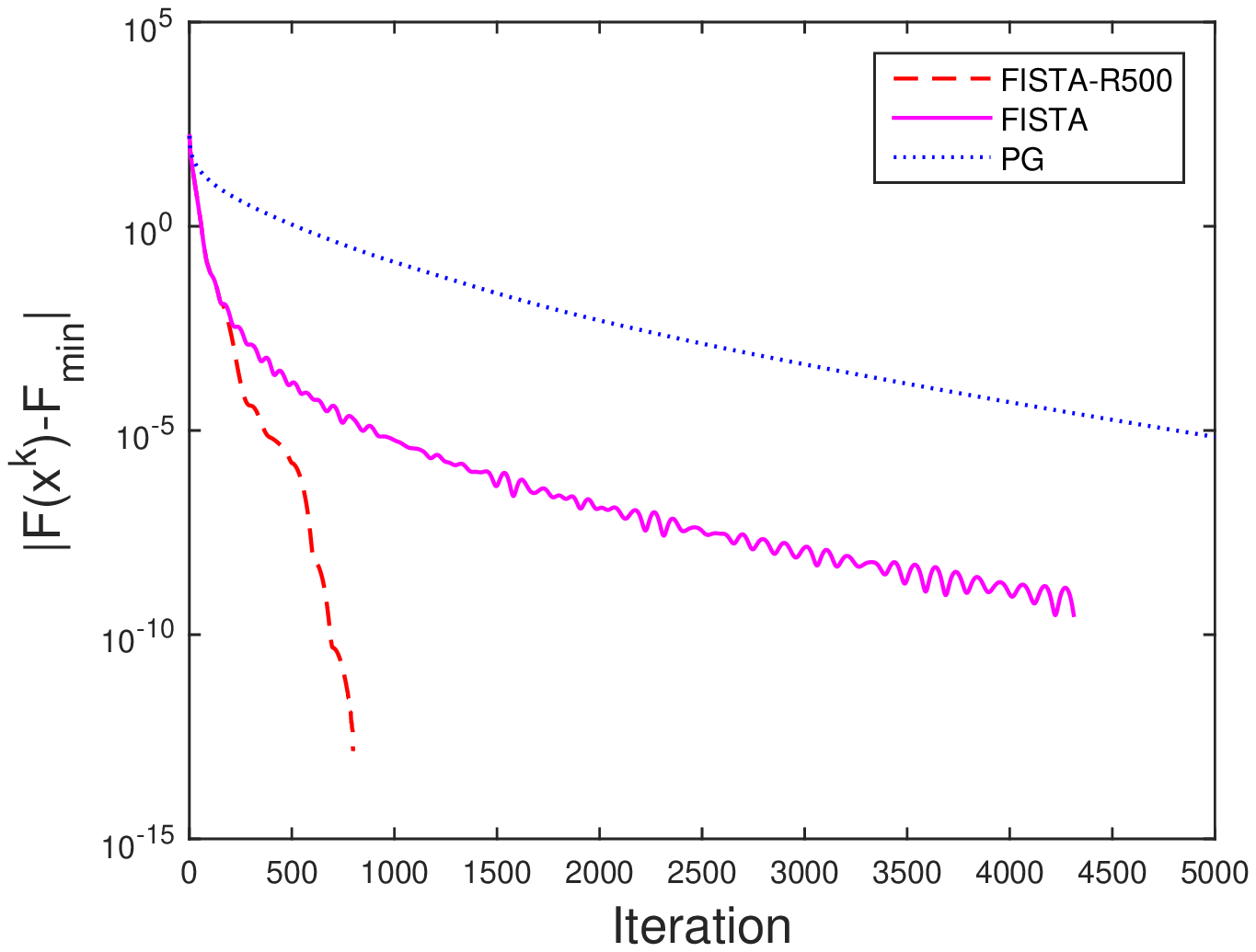}}
\end{figure}

\subsection{LASSO}
In this subsection, we consider the LASSO:
\begin{equation}\label{lasso}
v_{_{\rm ls}}:= \mathop{\min}_{x\in {{\mathbb{R}}^{n}}} \frac{1}{2}\|Ax-b\|^{2}+\lambda \|x\|_1,
\end{equation}
where $A\in {\mathbb{R}}^{m\times n}$ and $b\in {\mathbb{R}}^{m}$. We observe that \eqref{lasso} is in the form of \eqref{compro} with
\begin{equation}\label{lassofg}
f(x)=\frac{1}{2}\|Ax-b\|^{2},~~~~g(x)=\lambda\|x\|_1.
\end{equation}
It is clear that $f$ has a Lipschitz continuous gradient and $f+g$ has compact lower level sets. Thus, we can apply Algorithm 1 to solving \eqref{lasso}. Moreover, in view of the discussion following Assumption~\ref{errorbound}, Assumption~\ref{errorbound} is satisfied for \eqref{lassofg}. Hence, according to Corollary~\ref{propFISTA}, one should observe $R$-linear convergence of both the sequences $\{x^k\}$ and $\{F(x^k)\}$ generated by FISTA with restart. Finally, it is not hard to show that $\nabla f$ has a Lipschitz continuity modulus of $\lambda_{\max}(A^\top A)$. In view of this, in the algorithms below, we take $L=\lambda_{\max}(A^\top A)$ and $l=0$.

Before describing our numerical experiments, we recall that $f(x)= h(Ax)= \frac{1}{2}\|Ax-b\|^{2}$, where $h(v)=\frac{1}{2}\|v-b\|^2$. The conjugate function of $h$ can then be easily computed as $h^*(u) := \sup_{v\in \R^m}\{u^\top v - h(v)\}= \frac{1}{2}\|u\|^2+ b^\top u$. Hence, the dual problem of \eqref{lasso} is given by
\begin{equation}\label{dual_ls}
\begin{array}{rl}
\max\limits_{u \in \R^m}&d_{_{\rm ls}}(u) := -\frac{1}{2}\|u\|^2 - b^\top u\\
\mathrm{s.t.}& \|A^\top u\|_\infty \le \lambda.
\end{array}
\end{equation}
It can be shown that the optimal values of \eqref{lasso} and \eqref{dual_ls} are the same, and moreover, an optimal solution of \eqref{dual_ls} exists; see, for example, \cite[Theorem~3.3.5]{BL06}.
This dual problem will be used in developing termination criterion for our algorithms below.

Now we perform numerical experiments to study Algorithm 1 under the same three choices of $\{\beta_k\}$ as in the previous subsection. We choose $\lambda = 5$ in \eqref{lasso}, initialize all three algorithms at the origin and use the duality gap to terminate the algorithms. Specifically, as in the previous subsection, we make use of the fact that for any optimal solution $\bar{x}$ of \eqref{lasso}, $\nabla h(A\bar x)$ is an optimal solution of \eqref{dual_ls}. Hence, we define
\begin{equation*}
u^k=\min\left\{1,\frac{\lambda}{\left\|A^\top \nabla h(Ax^k)\right\|_\infty}\right\}\nabla h(Ax^k),
\end{equation*}
and terminate the algorithms once the duality gap is small, i.e.,
\begin{equation*}
\frac{|f(x^k) + g(x^k) - d_{_{\rm ls}}(u^k)|}{\max\{f(x^k) + g(x^k),1\}} \le 10^{-6}.
\end{equation*}
We also terminate them when the number of iterations hits $5000$.

The problems used in our experiments are generated as follows. For each $(m,n,s) = (300,3000,30)$, $(500,5000,50)$ and $(800,8000,80)$, we generate an $m\times n$ matrix $A$ with i.i.d. standard Gaussian entries. We then choose a support set $T$ of size $s$ uniformly at random, and generate an $s$-sparse vector $\hat x$ supported on $T$ with i.i.d. standard Gaussian entries. The vector $b$ is then generated as $b = A\hat x + 0.01\tilde{e}$, where $\tilde{e}$ has standard i.i.d. Gaussian entries.

The computational results are presented in Figures~\ref{ls3000}, \ref{ls5000} and \ref{ls8000}. We plot $\|x^k-x^*\|$ against the number of iterations in part (a) of each figure, where $x^*$ denotes the approximate solution obtained at termination of the respective algorithm; additionally, we plot $|F(x^k) - F_{\rm min}|$ against the number of iterations in part (b) of each figure, where $F_{\rm min}$ denotes the minimum of the three objective values obtained from the three algorithms. As in the previous subsection, we see from the figures that both $\{x^k\}$ and $\{F(x^k)\}$ generated by FISTA with both fixed and adaptive restart schemes are $R$-linearly convergent, which conforms with our theory. Additionally, the algorithm with restart performs better than FISTA and the proximal gradient algorithm.
\begin{figure}[h]
\centering
\caption{$l_1-ls:~n=3000,m=300, s=30$}\label{ls3000}
\subfigure[]{\includegraphics[height=3.8cm]{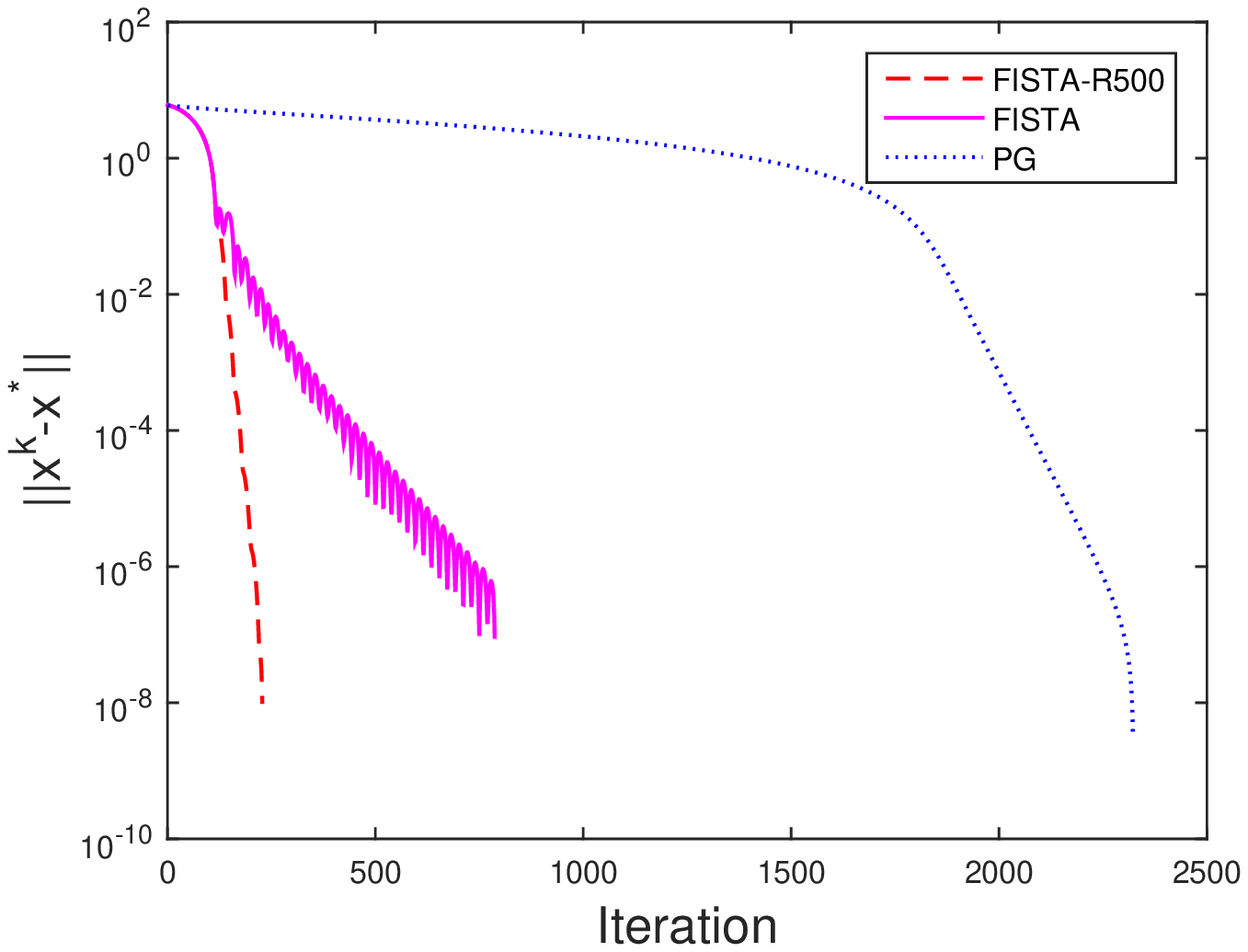}}
\subfigure[]{\includegraphics[height=3.8cm]{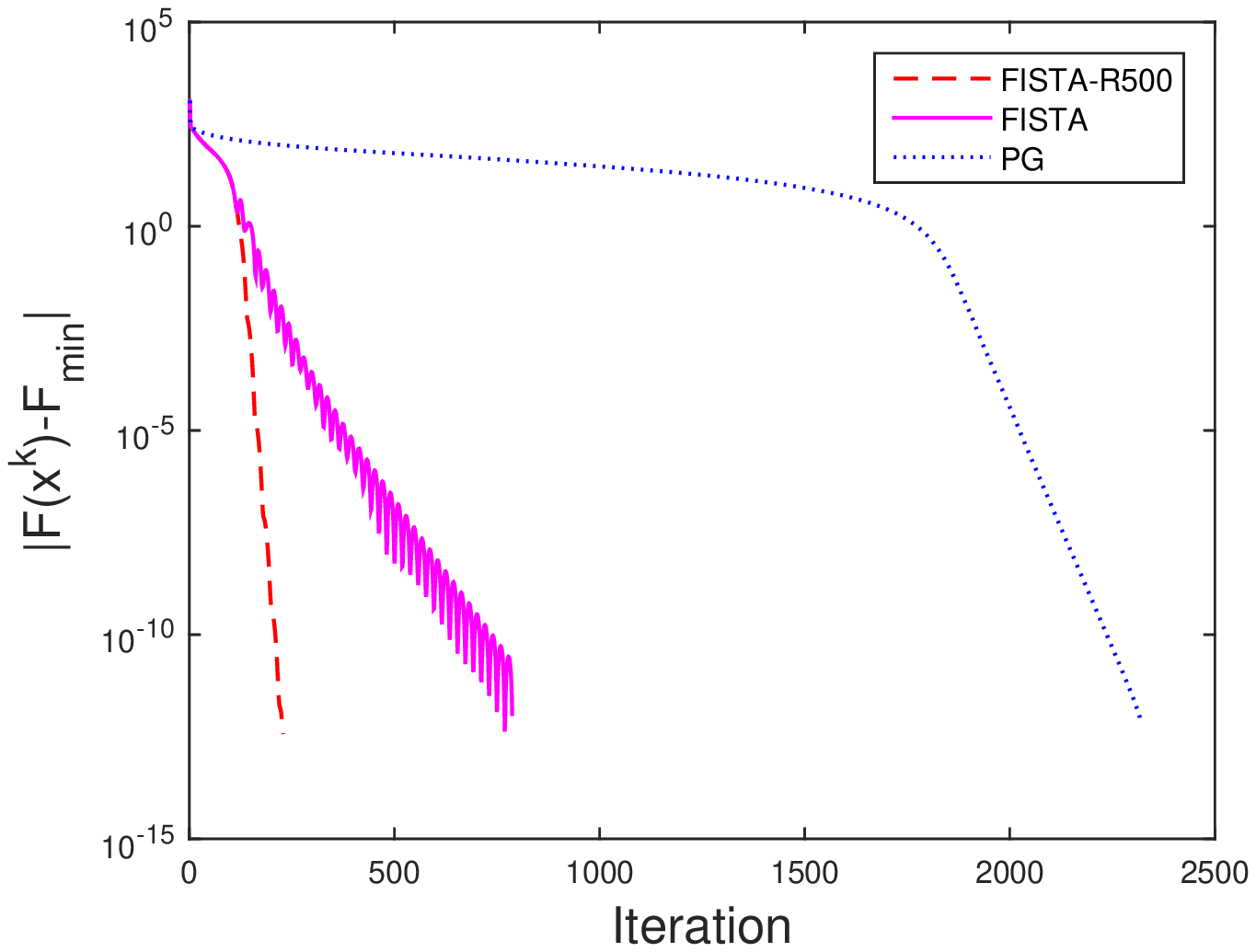}}
\end{figure}

\begin{figure}[h]
\centering
\caption{$l_1-ls:~n=5000,m=500, s=50$}\label{ls5000}
\subfigure[]{\includegraphics[height=3.8cm]{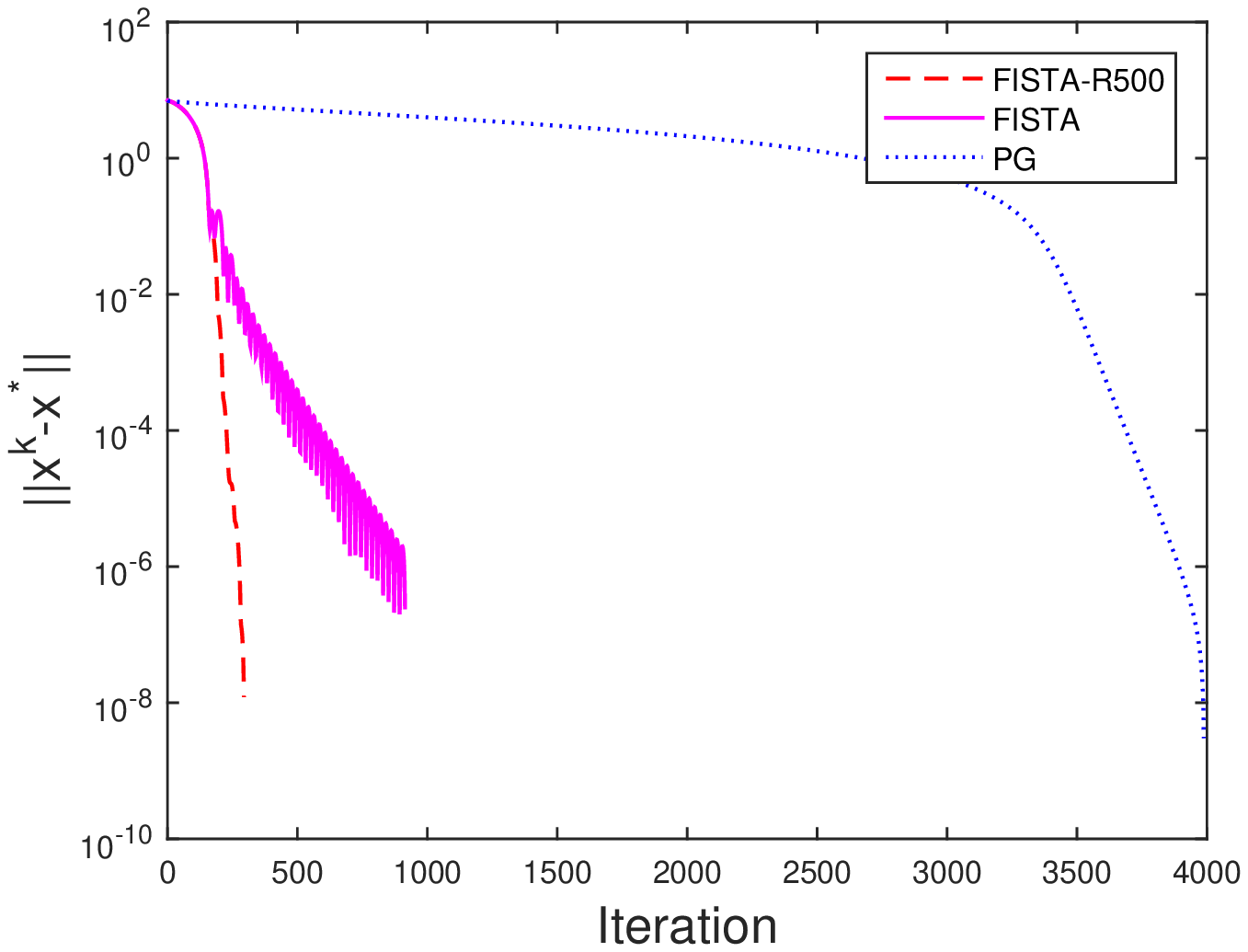}}
\subfigure[]{\includegraphics[height=3.8cm]{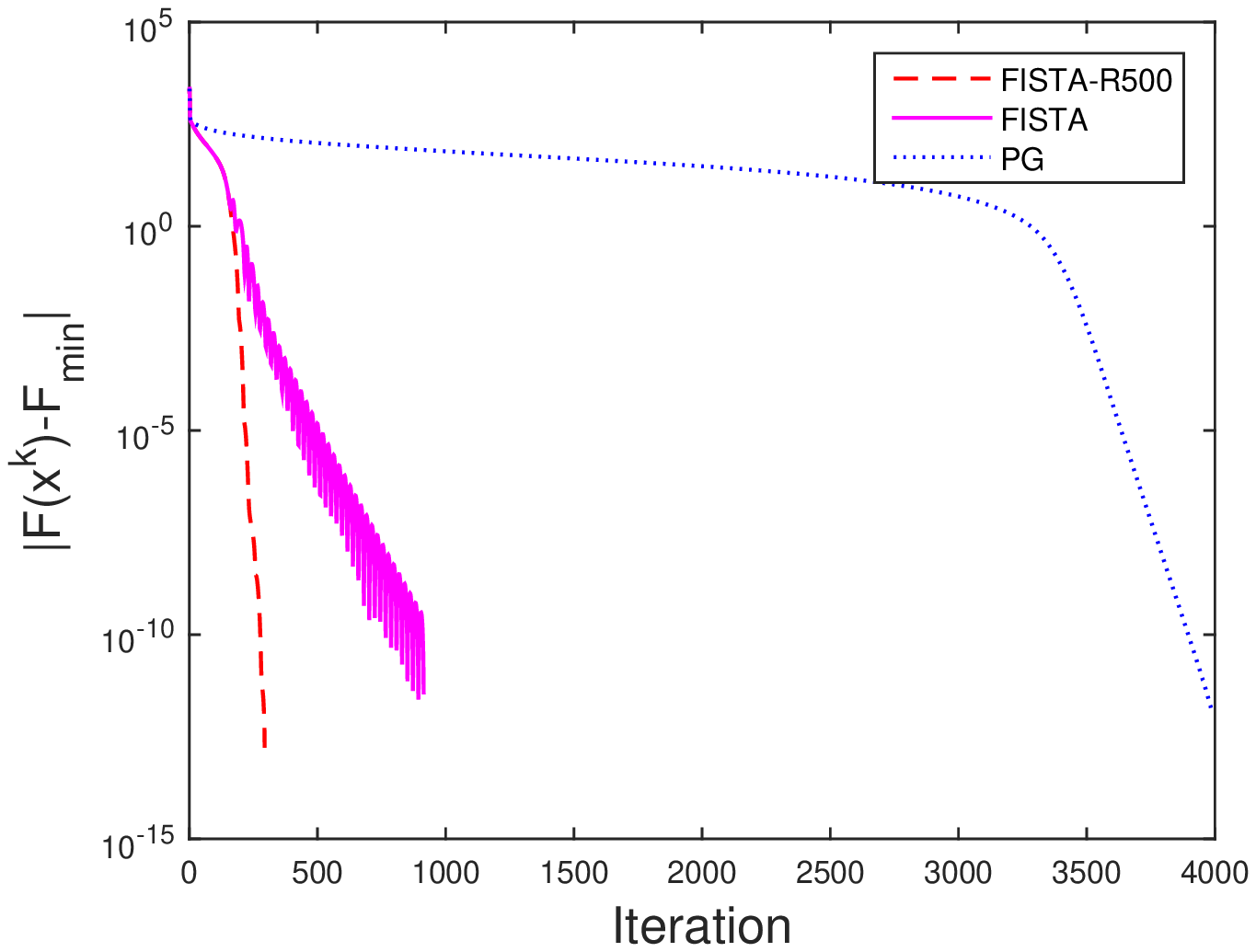}}
\end{figure}

\begin{figure}[h]
\centering
\caption{$l_1-ls:~n=8000,m=800, s=80$}\label{ls8000}
\subfigure[]{\includegraphics[height=3.8cm]{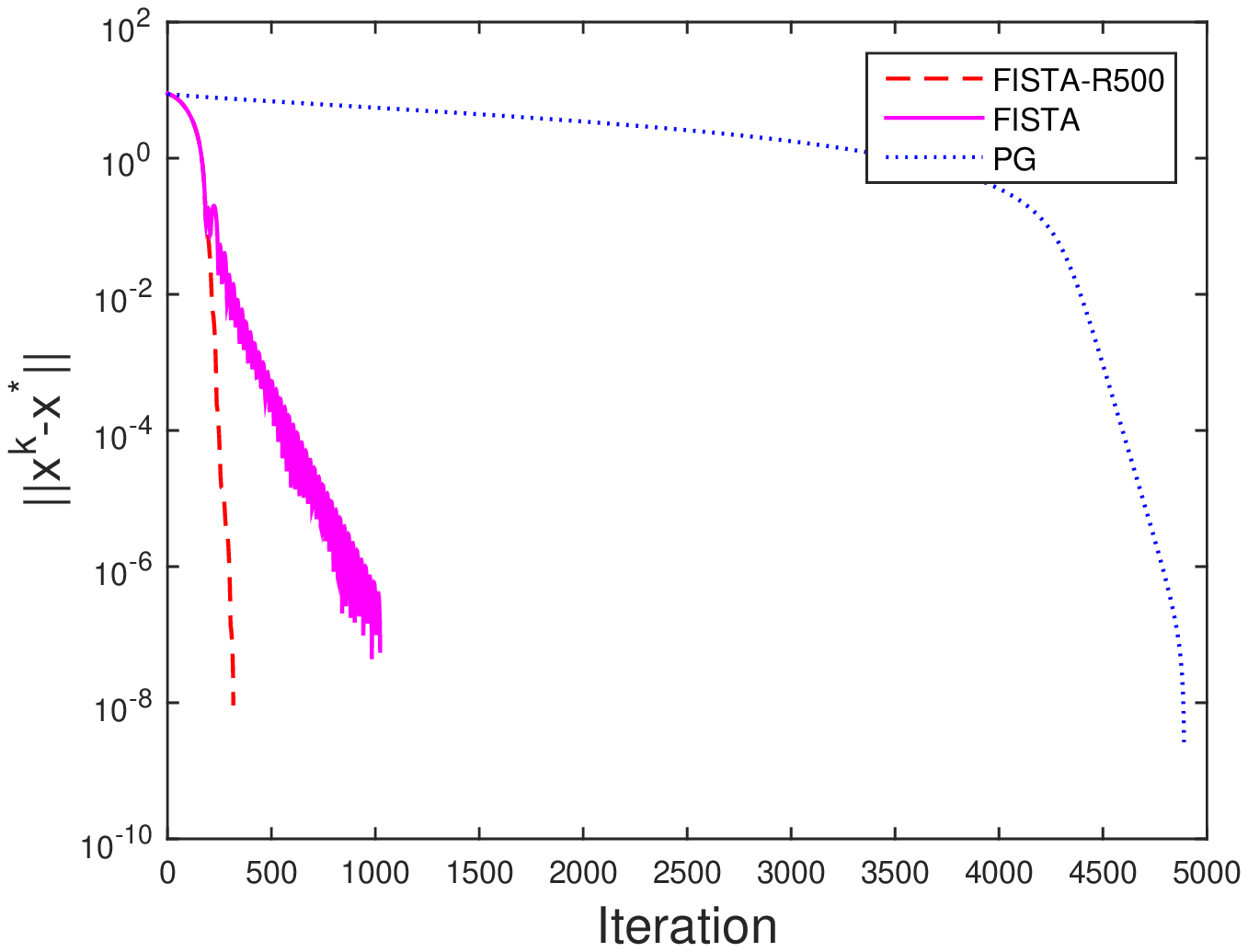}}
\subfigure[]{\includegraphics[height=3.8cm]{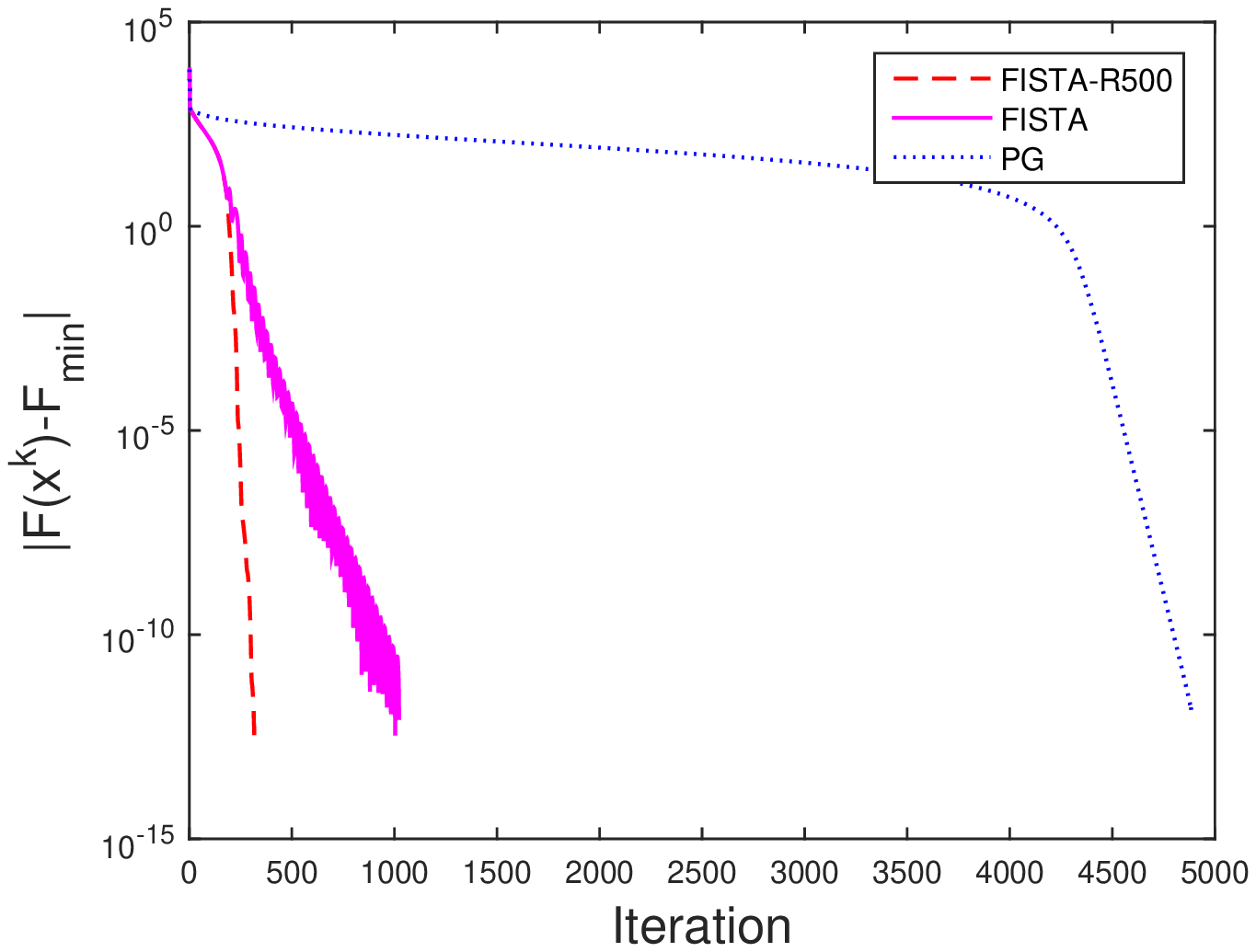}}
\end{figure}

\subsection{Nonconvex quadratic programming with simplex constraints}\label{sec4.2}
In this subsection, we look at problems of the following form, which are possibly nonconvex:
\begin{eqnarray}\label{NQS}
\begin{aligned}
&\min\limits_{x\in {{\mathbb{R}}^{n}}} ~\frac{1}{2} x^{\top}Ax-b^{\top}x\\
&~~\mathrm{s.t.} ~~e^\top x =s,~~x\ge 0,
\end{aligned}
\end{eqnarray}
where $A\in {{\mathbb{R}}^{n\times n}}$ is a symmetric matrix that is not necessarily positive semidefinite, $b \in \R^n$ and $s$ is a positive number. This is an example of nonconvex quadratic programming problems, which is an important class of problems in global optimization \cite{Chen2013, Gibbons1997, Ibaraki1988, Markowitz1952}. Notice that one can rewrite \eqref{NQS} in the form of \eqref{compro} by defining
\begin{equation}\label{NQSfg}
f(x)= \frac{1}{2} x^{\top}Ax-b^{\top}x, ~~~~g(x)=\delta_{\mathcal{S}}(x),
\end{equation}
where $\mathcal{S}=\left\{x\in \R^{n}:\;  e^\top x=s, ~x\ge 0 \right\}$. Moreover, it is clear that $f$ has a Lipschitz continuous gradient and $f + g$ is level bounded. Hence, Algorithm 1 can be applied to solving \eqref{NQS}. Furthermore, from the discussion following Assumption~\ref{errorbound}, Assumption~\ref{errorbound} is satisfied for \eqref{NQSfg}. Consequently, according to Theorem~\ref{Linearcon}, one should expect to see $R$-linear convergence of both the sequences $\{x^k\}$ and $\{F(x^k)\}$ generated by Algorithm 1 when $\bar\beta < \sqrt{\frac{L}{L+l}}$. Finally, since $A = A_1 - A_2$, where $A_1$ and $-A_2$ are the projections of $A$ onto the cone of positive semidefinite matrices and the cone of negative semidefinite matrices, respectively, we see that $f = f_1 - f_2$, where $f_1(x) = \frac12 x^\top A_1 x - b^\top x$ and $f_2(x) = \frac12 x^\top A_2 x$. In view of this, in our experiments below, we set $L = \max\{\lambda_{\max}(A),|\lambda_{\min}(A)|\}$ and $l = |\lambda_{\min}(A)|$ so that $L$ and $l$ are the Lipschitz continuity moduli of $\nabla f_1$ and $\nabla f_2$, respectively, and $L\ge l$.

Now we perform numerical experiments to study Algorithm 1 with two choices of $\{\beta_k\}$: $\beta_k \equiv 0$ (PG) and $\beta_k \equiv 0.98\sqrt{\frac{L}{L + l}}$ (PG$_{\rm e}$). In addition, we also perform the same experiments on FISTA.\footnote{We would like to point out that  FISTA applied to the nonconvex problem \eqref{NQS} is not known to converge, unlike the other two algorithms which have convergence guarantee by our theory.} We initialize all three algorithms at the origin, and terminate them when the successive changes of the iterates are small, i.e.,
\[
\frac{\|x^k - x^{k-1}\|}{\max\{\|x^k\|,1\}} \le 10^{-6}.
\]
We also terminate the algorithms when the number of iterations hits $5000$.

Our test problem is generated as follows. We generate a $2000\times 2000$ matrix $D$ with i.i.d. standard Gaussian entries. We then generate a symmetric matrix $A = D + D^\top$. Finally, the vector $b$ is generated with i.i.d. standard Gaussian entries, and $s$ is generated as $\max\{1,10t\}$, with $t$ chosen uniformly at random from $[0,1]$.

The computational results are presented in Figure~\ref{fig:nonconvex}. We plot $\|x^k - x^*\|$ against the number of iterations in Figure~\ref{fig:nonconvex} (a), where $x^*$ denotes the approximate solution obtained at termination of the respective algorithm; in addition, we plot $|F(x^k) - F_{\min}|$ against the number of iterations in Figure~\ref{fig:nonconvex} (b), with $F_{\min}$ being the minimum of the three objective values obtained from the three algorithms. We can see from Figure~\ref{fig:nonconvex} (a) that the sequence $\{x^k\}$ generated by Algorithm 1 with $\beta_k \equiv 0.98\sqrt{\frac{L}{L + l}}$ is $R$-linearly convergent, which conforms with our theory. However, from Figure~\ref{fig:nonconvex} (b), one can see that not all the algorithms are approaching $F_{\min}$. This is likely because the iterates generated by the algorithm got stuck at local minimizers.
\begin{figure}[h]
\centering
\caption{$Nonconvex ~Quadratic~ Problem$}\label{fig:nonconvex}
\subfigure[]{\includegraphics[height=3.8cm]{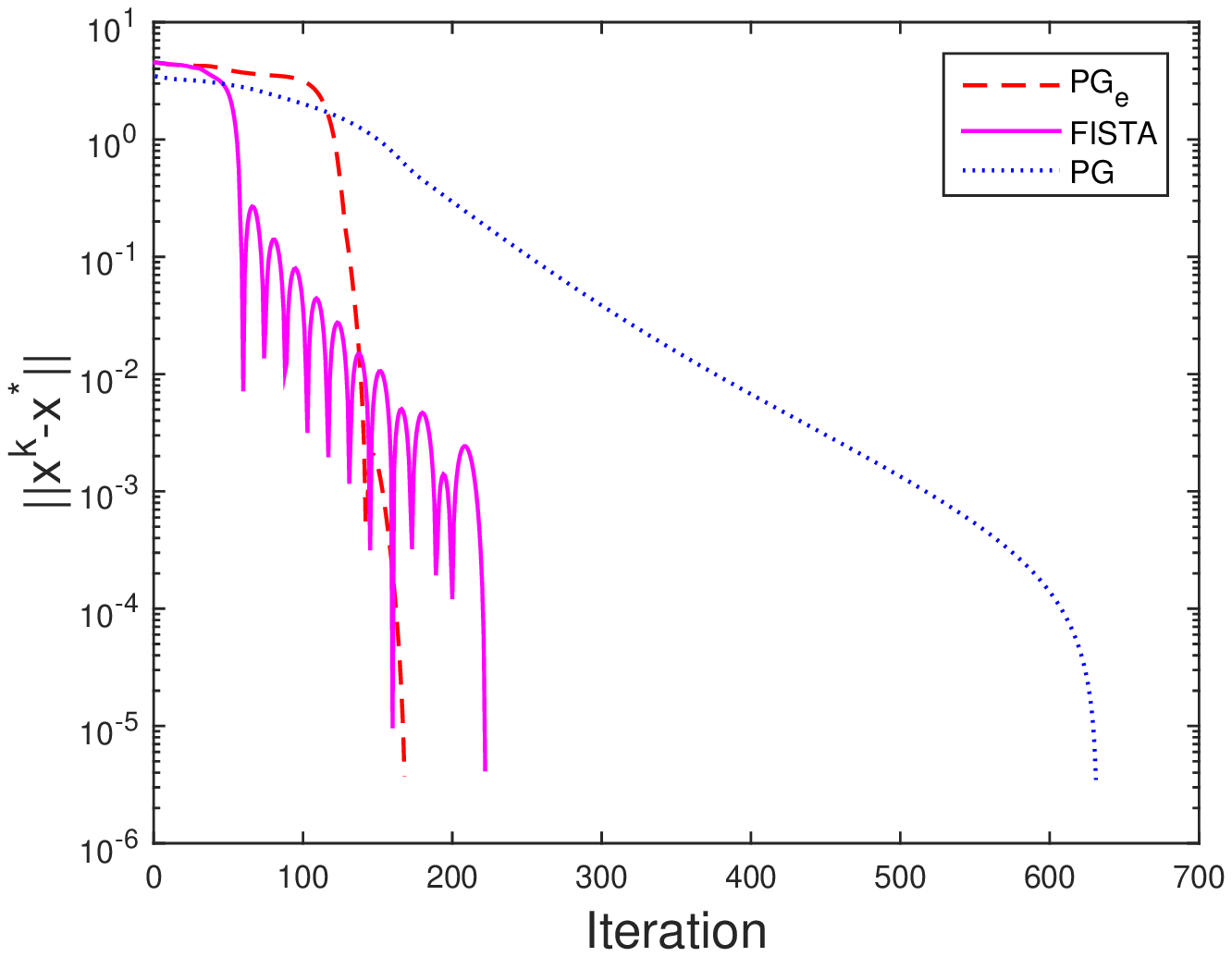}}
\subfigure[]{\includegraphics[height=3.8cm]{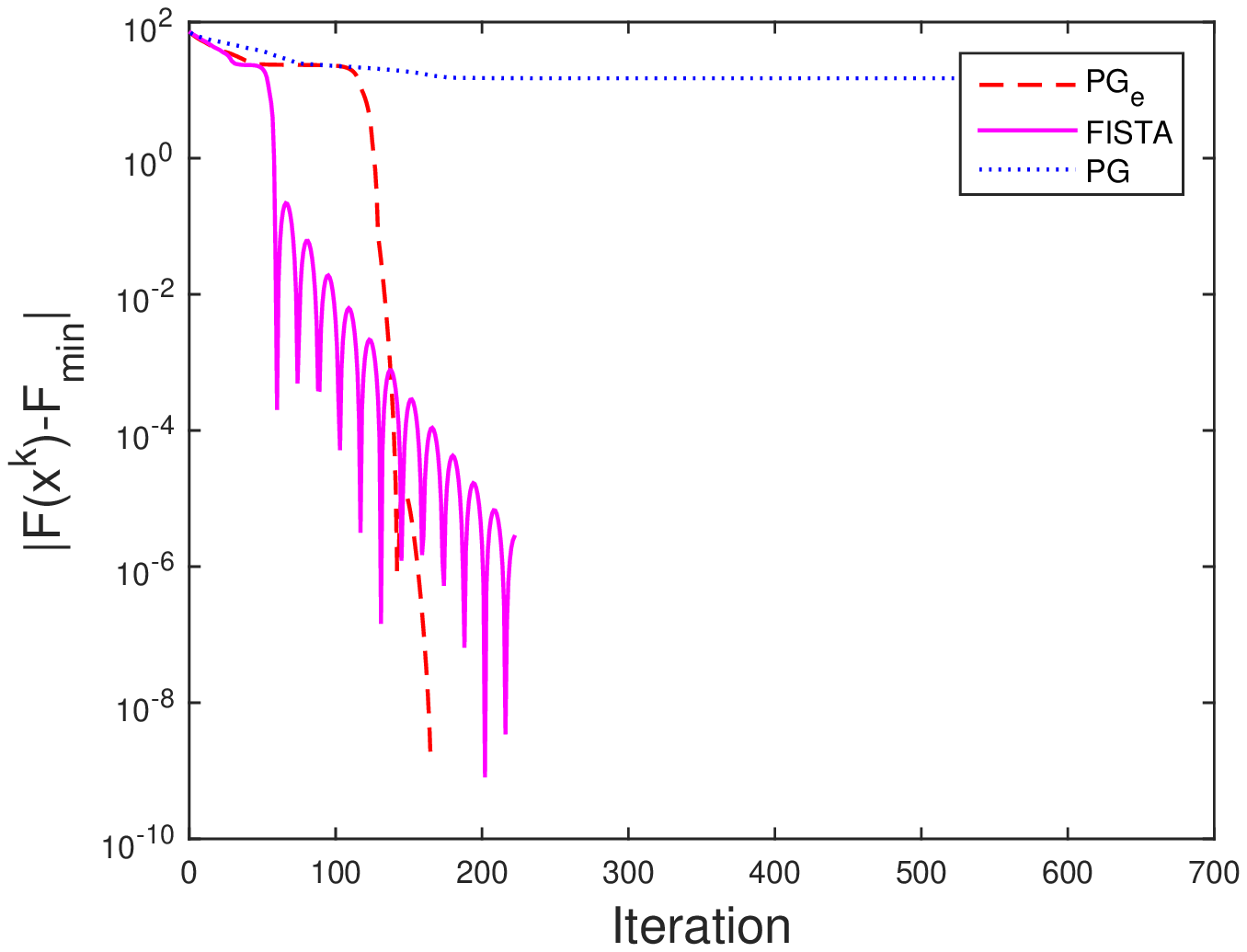}}
\end{figure}

To further evaluate the quality (in terms of function values at termination) of the approximate solution obtained by the algorithms, we perform a second experiment. In this second experiment, we generate random instances as follows: we generate an $n\times n$ matrix $D$ with i.i.d. standard Gaussian entries and symmetrize it to form $A = D + D^\top$; moreover, we generate a vector $b$ with i.i.d. standard Gaussian entries, and an $s=\max\{1,10t\}$, where $t$ is chosen uniformly at random from $[0,1]$.

In our test, for each $n=500$, $1000$, $1500$, $2000$ and $2500$, we generate $50$ random instances as described above. The computational results are reported in Table~\ref{t1}, where we present the number of iterations averaged over the $50$ instances for each $n$ (iter), and the function value at termination (fval), also averaged over the $50$ instances. One can see that while Algorithm 1 with $\beta_k \equiv 0.98\sqrt{\frac{L}{L + l}}$ (i.e., PG$_{\rm e}$) is always the fastest algorithm, the function values obtained can be slightly compromised for some instances.

\begin{table}[h]
\small
\caption{Comparing PG$_{\rm e}$, FISTA and PG on random instances.}\label{t1}
\begin{center}
\begin{tabular}{|c||c|c||c|c||c|c|}  \hline
\multicolumn{1}{|c||}{} & \multicolumn{2}{c||}{PG$_{\rm e}$} & \multicolumn{2}{c||}{FISTA} & \multicolumn{2}{c|}{PG}
\\ 
$n$& {\rm iter} & {\rm fval} & {\rm iter} & {\rm fval}  & {\rm iter} & {\rm fval}
\\ \hline
   500 &   120 & $-56.02$ &   175 & $-56.90$ &   322 & $-57.96$ \\
  1000 &   171 & $-69.77$ &   274 & $-66.79$ &   636 & $-66.93$ \\
  1500 &   166 & $-66.29$ &   270 & $-63.71$ &   560 & $-65.29$ \\
  2000 &   215 & $-80.72$ &   271 & $-80.43$ &   635 & $-81.21$ \\
  2500 &   284 & $-81.70$ &   359 & $-80.13$ &   813 & $-83.81$ \\
\hline
\end{tabular}
\end{center}
\normalsize
\end{table}

\section{Conclusion}\label{sec5}
In this paper, we study the proximal gradient algorithm with extrapolation for solving a class of nonconvex nonsmooth optimization problems. Based on the error bound condition, we establish the $R$-linear convergence of both the sequence $\{x^k\}$ generated by the algorithm and the corresponding sequence of objective values $\{F(x^k)\}$ if the extrapolation coefficients are below the threshold $\sqrt{\frac{L}{L+l}}$. We further demonstrate that our theory can be applied to analyzing the convergence of FISTA with the fixed restart scheme for convex problems.
Finally, we perform some numerical experiments to illustrate our results.

{\bf Acknowledgments.} We would like to thank the anonymous referees for their insightful comments and helpful suggestions that helped improve the manuscript.

\end{document}